%% file: paper.tex
\documentclass[manuscript]{acmart}

\AtBeginDocument{%
  }

\input{setup}

\setcopyright{acmlicensed}
\copyrightyear{2025}
\acmYear{2025}
\acmDOI{XXXXXXX.XXXXXXX}

\acmJournal{JACM}
\acmVolume{37}
\acmNumber{4}
\acmArticle{111}
\acmMonth{8}

\begin{document}

\title{A Model-Based Derivative-Free Optimization Algorithm for Partially Separable Problems}

\author{Yichuan Liu}
\email{liuyichuan2020@outlook.com}
\orcid{0009-0001-2692-9454}
\affiliation{%
  \institution{School of Mathematical Sciences, Fudan University}
  \city{Shanghai}
  \country{China}
  \postcode{200433}
}
\author{Yingzhou Li}
\email{yingzhouli@fudan.edu.cn}
\orcid{0000-0003-1852-3750}
\affiliation{%
  \institution{School of Mathematical Sciences, Fudan University}
  \city{Shanghai}
  \country{China}
  \postcode{200433}
}
\affiliation{
  \institution{Shanghai Key Laboratory for Contemporary Applied Mathematics}
  \city{Shanghai}
  \country{China}
  \postcode{200433}
}
\affiliation{
  \institution{Key Laboratory of Computational Physical Sciences (MOE)}
  \city{Shanghai}
  \country{China}
  \postcode{200433}
}
\authornote{This work is supported by the National Natural Science
Foundation of China (NSFC) under grant numbers 12271109, the Science and
Technology Commission of Shanghai Municipality (STCSM) under grant numbers
22TQ017 and 24DP2600100, and the Shanghai Institute for Mathematics and
Interdisciplinary Sciences (SIMIS) under grant number SIMIS-ID-2024-(CN).}

\renewcommand{\shortauthors}{Y. Liu et al.}

\begin{abstract}
    We propose \texttt{UPOQA}, a derivative-free optimization algorithm for partially separable unconstrained problems, leveraging quadratic interpolation and a structured trust-region framework. By decomposing the objective into element functions, \texttt{UPOQA} constructs underdetermined element models and solves subproblems efficiently via a modified projected gradient method. Innovations include an approximate projection operator for structured trust regions, improved management of elemental radii and models, a starting point search mechanism, and support for hybrid black-white-box optimization, etc. Numerical experiments on 85 \texttt{CUTEst} problems demonstrate that \texttt{UPOQA} can significantly reduce the number of function evaluations. To quantify the impact of exploiting partial separability, we introduce the speed-up profile to further evaluate the acceleration effect. Results show that the speed-up of \texttt{UPOQA} over baselines is less significant in low-precision scenarios but becomes more pronounced in high-precision scenarios. Applications to quantum variational problems further validate its practical utility.
\end{abstract}




\maketitle

\section{Introduction}

Derivative-free optimization (DFO)~\cite{Intro2DFO} methods, also known as black-box optimization (BBO) or zero-order optimization methods, are designed to find the minimum or maximum of an objective function when its derivatives are not available. For complex problems, the objective function may not be expressed as a closed-form analytical expression but rather as a black box that only allows function evaluations. Examples include cases where function values are obtained from simulation or experimental observations, or from the output of software packages with unknown internal implementations. Such objective functions often incur high evaluation costs, yield noisy values, or both.

DFO algorithms can be broadly categorized into five classes. The first class comprises \emph{direct search} and \emph{pattern search} methods~\cite{Intro2DirectSearch, Intro2PatternSearch}, which select the next iteration point by comparing the relative merits of function values~\cite{DirectSearchDef}. These methods typically employ specific geometric patterns for point selection. They impose minimal requirements on the properties of the objective function, are computationally simple, and less prone to local minima, but often exhibit slow convergence. The second class consists of \emph{line-search-based} methods, with Powell's method~\cite{PowellMethod} being a representative example that performs line searches along conjugate direction sets. The third class includes \emph{heuristic search} methods, which explore the solution space following specific heuristic rules to find approximate optima. Common algorithms in this category are simulated annealing~\cite{SA} and evolutionary algorithms~\cite{ISRES, ESCH}, with \texttt{CMA-ES} being a notable example~\cite{CMA-ES}. The fourth class is \emph{model-based derivative-free optimization}~\cite{Intro2MBDFO, ConnMBDFO, PDFO, pybobyqa}. These methods construct a surrogate model through interpolation or regression using known function values, approximating the objective function locally. The model is then optimized and updated within a trust-region framework to solve the original problem. Examples include Powell's \texttt{COBYLA}~\cite{COBYLA} (based on first-order models) and \texttt{BOBYQA}~\cite{BOBYQA}, \texttt{NEWUOA}~\cite{NEWUOAsoftware} (based on second-order models), etc. Surrogate models may also adopt other forms, such as radial basis functions~\cite{RadialBasisDFO} and moving ridge functions (e.g., \texttt{OMoRF}~\cite{OMoRF}). The fifth class comprises \emph{finite-difference-based} methods\cite{FDLongIntroAndTest}, which approximate the true gradient using numerical differentiation and then employ gradient-based optimization algorithms.

In DFO, the primary challenge lies in the lack of structural knowledge about the problem. When the problem lacks smoothness, convexity, or is subject to significant noise, many DFO methods essentially sample the parameter space according to certain patterns. Such approaches are prone to the curse of dimensionality when dealing with large-scale problems. Even if smoothness is assumed and surrogate models are constructed to accelerate convergence, the computational cost of solving these problems may still be prohibitive~\cite{CMA-ESCurseDim}. To address this, many researchers have considered structured DFO problems, such as nonlinear least squares~\cite{powell1965LSDFOmethod, pybobyqa, DFLS, DFO-LS}, problems with sparse Hessian matrices~\cite{DFOforSparseHess, RN93, CMA-ESCurseDim} and partially separable problems~\cite{PSRandomDirectSearch, PSMeshPatternSearch, PS-CMAES, PSDFOThesis, PSNaiveDFO, PSParticleSwarm}. Exploiting these structures can significantly reduce the number of function evaluations required for convergence. 

This paper focuses particularly on unconstrained partially separable problems~\cite{PSorigin}, which take the form:  
\begin{equation}\label{eq:PSProb}  
    \min_{x\in \Real^n} f(x) = \sum_{i=1}^{q} f_i(x),  
\end{equation}  
where \( f_1, \ldots, f_q \) are referred to as \emph{element functions} (or simply \emph{elements}). For each \( i = 1, \ldots, q \), there exists a subspace \( \mathcal{N}_i \subset \mathbb{R}^n \) such that for any \( w \in \mathcal{N}_i \) and \( x \in \mathbb{R}^n \), the following holds:  
\begin{equation*}  
    f_i(x + w) = f_i(x).  
\end{equation*}  

In contrast to the elements, \( f \) is called the \emph{overall function}. Let \( \mathcal{R}_i \triangleq \mathcal{N}_i^\perp \), and denote \( n_i \triangleq \dim \mathcal{R}_i \) as the \emph{elemental dimension} of \( f_i \). Each \( f_i \) is essentially a function defined on \( \mathcal{R}_i \). The corresponding projection is denoted as \( P_{\mathcal{R}_i}\colon \mathbb{R}^n \to \mathcal{R}_i \), ensuring that for any \( x \in \mathbb{R}^n \), \( f_i(x) \equiv f_i(P_{\mathcal{R}_i}(x)) \). Without loss of generality, we assume \( \operatorname{span}\left(\cup_{i = 1}^q \mathcal{R}_i\right) = \mathbb{R}^n \), as otherwise, redundant variables would not affect the value of \( f \).

Partially separable structures frequently arise in optimal control problems, the discretization of partial differential equations (PDEs) via methods like finite elements, and of other variational problems~\cite{PSPartitionedVMU}. In optimal control, problems often consist of multiple loosely coupled subsystems distributed across space and time. Similarly, problems stemming from domain decomposition techniques applied to PDE discretizations exhibit partial separability, where each variable primarily interacts with its neighboring regions. In quantum computing, many variational quantum problems also demonstrate such structures, with each subfunction represented by a quantum circuit~\cite{QOMM, QOMM2}. Partially separable structures have also been exploited in machine learning, where the training of neural networks constitutes a large-scale nonlinear programming problem. Bigeon, Orban, and Raynaud developed \texttt{PartiallySeparableNLPModels.jl}~\cite{PSNLPModels} to automatically detect and exploit such structures within partitioned quasi-Newton solvers, and proposed limited-memory variants for high-dimensional element functions~\cite{PSLimitedMemoryQN}. Raynaud, Orban, and Bigeon further introduced a partially-separable loss and a partitioned neural network architecture to enable distributed training~\cite{PSLoss}, together with a limited-memory partitioned quasi-Newton optimizer~\cite{PLSR1}. A comprehensive account is given in their survey~\cite{PSSurvey}.

In practice, the more commonly used concept is \emph{coordinate partial separability}, which further requires that each \( f_i \) depends only on a subset of variables \( \left\{x_j \mid j \in \mathcal{I}_i = \{j_1, \ldots, j_{n_i}\} \right\} \), referred to as the \emph{elemental variables} of \( f_i \). In this case, \( \mathcal{R}_i = \operatorname{span}(\{e_j \mid j \in \mathcal{I}_i\}) \), and we denote $P_{\mathcal{R}_i}(x)$ as $x^{\mathcal{I}_i}$. Since \( f_i(x) \) depends solely on \( x^{\mathcal{I}_i} \), we will, for simplicity, no longer distinguish between the notations \( f_i(x^{\mathcal{I}_i}) \) and \( f_i(x) \) in subsequent discussions. Due to its prevalence and algorithmic convenience, unless explicitly stated otherwise, all references to partial separability in this paper refer to coordinate partial separability. We further assume that each element function can be evaluated independently. This assumption aligns with the practical context of most problems, as when each \( f_i \) represents the contribution of a subsystem, these contributions should naturally be computable separately.

The term \emph{partial} in partial separability refers to the fact that the variables on which different element functions depend may overlap. The two extremes of this property are when all elements depend on completely different variables, such as $f(x) = \sum_{i=1}^q f_i(x_i)$, and when all elements depend on all variables, i.e., $f(x) = \sum_{i=1}^q f_i(x)$. However, most real-world problems lie between these two cases~\cite{TestPSProb1, PSforHisMatch, TestPSProb2}. In the \texttt{LANCELOT} package~\cite{LANCELOT}, the concept of partial separability is further generalized to \emph{group partial separability} and used to define the \emph{standard input format} (SIF) for test problems. Moreover, many common problem structures can be viewed as special cases of partially separable structures. Any twice continuously differentiable function with a sparse Hessian can be transformed into a partially separable form~\cite{PSorigin}, and nonlinear least squares problems $\min_{x} \sum_{i=1}^N r_i(x)^2$ are naturally partially separable.

The most obvious advantage of partial separability is that, given a trial point $x\in \Real^n$, one can always compute the value of any element function $f_i(x)$ individually without evaluating the full $f(x)$. Conversely, even if $f(x)$ must be computed, all element values $f_1(x),\ldots,f_q(x)$ can be obtained simultaneously, thereby significantly increasing the amount of information available. Various methods have been developed to exploit partial separability. When derivatives are available, a common approach is to employ partitioned Hessian updating schemes~\cite{PSPartitionedVMU, PSPartitionedVMUAddi, PSPartitionedBFGS, PSPartitionedBFGSPara, PSPartitionedPSB} in quasi-Newton methods, along with techniques such as element merging and grouping~\cite{PSPartitionedEleMerge, PSIncreOpt}, element-wise preconditioning~\cite{PSElePreCond}, and parallelization~\cite{PSPartitionedBFGSPara} to reduce the computational cost of Newton steps or conjugate gradients. Conn et al. proposed a trust-region framework~\cite{PSConnTrustRegion} for solving partially separable problems, which was later extended by Shahabuddin~\cite{PSDFOThesis} with three algorithm variants based on different methods for solving the trust-region subproblems. Another approach is \emph{incremental optimization}~\cite{IncreOpt, PSIncreOpt}, where only a small subset of elements is optimized in each iteration. 

When derivatives are unavailable, existing approaches can be broadly categorized into \emph{bottom-up} and \emph{top-down} approaches, depending on whether they directly optimize the element functions or optimize the objective function. The \emph{bottom-up} approach primarily leverages the independent evaluability of elements to reconstruct information about the objective function from element values at relatively low cost, thereby solving problem (\ref{eq:PSProb}). A straightforward implementation evaluates elements at grid points in $\mathbb{R}^n$, e.g., $\{x=$ $\left.\sum_{i=1}^n c_i e_i \mid c_i \in \mathbb{Z}\right\}$, then assembles these evaluations into full function values. For instance, for ${f(x, y, z)}=f_1(x, y)+f_2(y, z)$, computing the values of $f_i$ at $\{0,1\}^2$ requires only 4 evaluations but yields all 8 values of $f$ over $\{0,1\}^3$. If elements are independent, evaluating each element $k$ times produces $k^q$ objective values from just $kq$ evaluations. Price and Toint~\cite{PSMeshPatternSearch} exploited this to develop a mesh-refinement-based direct search method that scales efficiently to thousands of dimensions. Another strategy exploits the structure within a random pattern-search framework, such as Porcelli and Toint's \texttt{BFO} method~\cite{PSRandomDirectSearch}. Its structured extension restricts the polling directions to the subspaces of elemental variables, so that only the affected elements need to be evaluated to detect decrease. It also provides an optional library of per-element interpolation models minimized within a trust region as a search step, though the main loop remains a random pattern search rather than a model-based trust-region algorithm.

An alternative approach is \emph{top-down}, where the algorithm primarily extracts information from complete evaluations of $f$ to guide the optimization. Model-based methods often fall into this category, as the subproblems formed by surrogate models typically generate points without a clear geometric pattern, making it unlikely to reconstruct the objective value by computing only a few element values. Relevant work includes Bouzarkouna et al.'s \texttt{p-sep lmm-CMA} method~\cite{PS-CMAES} and Conn and Toint's \texttt{PSDFO} method~\cite{PSNaiveDFO}. The \texttt{p-sep lmm-CMA} method is a variant of \texttt{lmm-CMA}~\cite{IMM-CMA}, where a meta-model (a second-order polynomial regression model) is built for each element function to guide the ranking of data points in \texttt{lmm-CMA}. Meanwhile, \texttt{PSDFO} is an interpolation-based trust-region method where each element $f_i$ is approximated by a quadratic interpolation model $m_i$. In both methods, the primary purpose of modeling elements is to construct an \emph{overall model}  
\begin{equation}\label{eq:OverallModel}  
    m(x) = \sum_{i=1}^q m_i(x),  
\end{equation}  
which is then used as a single surrogate model in the underlying algorithm for general optimization. Consequently, the information available to the algorithm mainly stems from feedback on the full function $f$ or the overall model $m$. The benefits of the \emph{top-down} approach are relatively indirect. Compared to using a single surrogate model, employing multiple element models primarily reduces the total modeling cost, since evaluating $f$ once yields all values $f_1,\ldots,f_q$, enabling simultaneous updates to all models in a single iteration. 

Generally speaking, model-based methods exhibit superior computational efficiency compared to direct search and pattern search approaches. In pursuit of efficiency, we adopt a \emph{top-down} strategy and develop a trust-region algorithm that employs quadratic interpolation models to approximate element functions. The algorithm is named \emph{Unconstrained Partially-separable Optimization by Quadratic Approximation} (\texttt{UPOQA}). The key features of \texttt{UPOQA} include:

\begin{enumerate}
    \item For interpolation modeling, \texttt{UPOQA} utilizes the underdetermined quadratic model based on Powell's \emph{derivative-free symmetric Broyden update}~\cite{PowellLeastFrobeniusUpdate}, incorporating Powell's techniques~\cite{PowellUpdateInverseKKT,PowellDevNEWUOA} to enable efficient model construction and maintenance with low algebraic complexity;
    \item For trust-region management, \texttt{UPOQA} assigns independent trust-region radii to each element, implements a modified projected gradient method for solving trust-region subproblems within structured trust regions, improves upon the existing radius adjustment strategy~\cite{PSDFOThesis}, and introduces a selective update mechanism for element models based on obtained interpolation points;
    \item In terms of accessibility, \texttt{UPOQA} provides a ready-to-use Python implementation with various features to enhance flexibility and robustness, including hybrid black-white-box optimization and restart mechanism. The open-source package is released on Github\footnote{See https://github.com/Chitius/upoqa. Version 1.0.1 was used for all the testing below.} under the GNU General Public License.
\end{enumerate}

The rest of the paper is organized as follows. Section \ref{sec:Upoqa_frame} presents the overall framework of \texttt{UPOQA}, including the designed approximate projection operator for structured trust-region problems, the management of trust-region radius, the criteria for selective element model updates, starting point search mechanism, and optional features (restart and hybrid optimization). Section \ref{sec:upoqa_experiment} conducts comprehensive numerical experiments on \texttt{CUTEst} test problems extracted via the \texttt{S2MPJ} tool~\cite{s2mpj} and test cases from the quantum variational problem (\ref{eq:quant_prob_intro}) to validate the algorithm's effectiveness. Section \ref{sec:conclusion} concludes the work and discusses potential future research directions.

\section{The \texttt{UPOQA} Algorithm}\label{sec:Upoqa_frame}

The \texttt{UPOQA} algorithm is based on Powell's trust-region framework~\cite{NEWUOAsoftware,BOBYQA,PDFO} for solving problem (\ref{eq:PSProb}). For each element $f_i$, the algorithm maintains an interpolation set $\mathcal{Y}_i$ of size at least $O(n_i)$, along with an underdetermined quadratic interpolation model~\cite{PowellLeastFrobeniusUpdate} constructed from these points:
\begin{equation}\label{eq:QuadModelAtItk}
    m_{k,i}(x_k^{\mathcal{I}_i} + s) = c_{k,i} + g_{k,i}^\top s + \frac{1}{2} s^\top H_{k,i} s,\quad s \in \Real^{n_i}.
\end{equation}
We refer to $m_{k,1}, \ldots, m_{k,q}$ as \emph{element models}, where $c_{k,i} \in \Real$, $g_{k,i} \in \Real^{n_i}$, and $H_{k,i} \in \Real^{n_i \times n_i}$ is a symmetric matrix. The model $m_{k,i}$ is expected to sufficiently approximate the objective function within a trust region of radius $\Delta_{k,i}$:
\begin{equation*}
    \mathcal{B}_i(\Delta_{k,i}) = \left\{x_k^{\mathcal{I}_i} + s\mid s\in \Real^{n_i},\ \|s\|_2 \leq \Delta_{k,i}\right\}.
\end{equation*}
Moreover, each model is equipped with its own trust-region radius  $\{\Delta_{k,i}\}_{i=1}^q$ to account for their varying reliability due to different accuracies.

During the execution of \texttt{UPOQA}, as the interpolation sets are updated, the element models are also updated by the \emph{derivative-free symmetric Broyden update}~\cite{PowellLeastFrobeniusUpdate}, which has been widely employed by Powell in algorithms such as \texttt{NEWUOA}~\cite{NEWUOAsoftware} and \texttt{BOBYQA}~\cite{BOBYQA}. It allows the algorithm to use only $O(n_i)$ interpolation points per model while updating the model with $O(n_i^2)$ algebraic complexity~\cite{PowellUpdateInverseKKT,PowellDevNEWUOA}. This surrogate model is widely recognized as an effective second-order approach  in derivative-free methods based on interpolation models and trust regions~\cite{Intro2DFO}.

Our work is tailored for partially separable problems. When the number of elements exceeds one, two key issues must be addressed. First, how to solve the resulting structured trust-region subproblem:  
\begin{eqnarray}\label{eq:PSTrustRegionProb0}  
    &\min\limits_{s\in \Real^n}\ & m_k(x_k + s),\\  
    &\text{s.t.}\ &  
    \|s^{\mathcal{I}_i}\|_2 \leq \Delta_{k,i},\quad i = 1,2,\ldots, q, \nonumber  
\end{eqnarray}  
where \( m_k \triangleq \sum_{i=1}^{q} m_{k,i} \) is the \emph{overall model} assembled from the element models \( \{m_{k,i}\}_{i=1}^q \). Second, how to evaluate the contribution of each element model to the reduction of the objective value in each iteration and accordingly adjust \( \{\Delta_{k,i}\}_{i=1}^q \). For the first issue, we employ a modified projected gradient method based on an approximate projection operator, which will be detailed in Section \ref{sec:structured_tr}. For the second issue, we directly adopt the combined separation criterion proposed by Shahabuddin~\cite[][Section 2.4]{PSDFOThesis}, with minor modifications in corner cases.  

Next, we briefly outline the general workflow of the \texttt{UPOQA} algorithm for solving problem (\ref{eq:PSProb}). At initialization, the algorithm first constructs the surrogate model \( m_{0,i} \) and the interpolation set \( \mathcal{Y}_{0,i} \) for each element \( f_i \). Then, using all interpolation points from \( \{\mathcal{Y}_{0,i}\}_{i=1}^q \), it seeks a better starting point \( x_0 \) via the method described in Section \ref{sec:upoqa_start}. If successful, the objective value \( f(x_0) \) will be lower than that of the user-provided point \( x_{\text{start}} \). In each iteration, \texttt{UPOQA} first solves the trust-region subproblem (\ref{eq:PSTrustRegionProb0}) using the modified projected gradient method from Section \ref{sec:structured_tr}, obtaining a trial step \( s_k \). The algorithm then decides whether to accept this step based on the magnitude of \( \|s_k\|_2 \).

For $s_k$ with large norms, the algorithm first computes a candidate point $\hat{x}_k = x_k + s_k$ and updates all models and interpolation point sets based on $\hat{x}_k$. When updating the model $m_{k,i}$ and the interpolation set $\mathcal{Y}_{k,i}$, a new interpolation point $\hat{x}_k^{\mathcal{I}_i}$ is added to $\mathcal{Y}_{k,i}$. However, this procedure is performed only if $\hat{x}_k^{\mathcal{I}_i}$ does not compromise the poisedness of the interpolation set $\mathcal{Y}_{k,i}$. For updating the trust-region radii $\left\{\Delta_{k,i}\right\}_{i=1}^q$, we employ the combined separation criterion proposed by Shahabuddin~\cite[][Section 2.4]{PSDFOThesis}. If the objective exhibits insufficient reduction at $\hat{x}_k$, the algorithm further checks whether a geometry-improving step should be taken. If the conditions for a geometry-improving step are not met, the trust-region resolution $\rho_k$ is reduced. The value of $\rho_k$ serves as a lower bound for all trust-region radii $\{\Delta_{k,i}\}_{i=1}^q$, indicating the search precision of the algorithm at the current iteration. For $s_k$ with small norms, the algorithm will reduce the trust-region radii and perform geometry-improvement steps. If such small-step events occur consecutively multiple times or no geometry improvement steps are triggered, the algorithm also reduces $\rho_k$.  
 
Algorithm \ref{alg:UPOQA} outlines the procedure of \texttt{UPOQA}. Steps unrelated to the partially separable structure closely resemble those in Powell-style model-based algorithms. For these aspects, we refer readers to the descriptions of these algorithms~\cite{NEWUOAsoftware,PowellDevNEWUOA,COBYQA,PDFO} or the implementation of \texttt{UPOQA} for further details.

\begin{algorithm}
    \caption{The \texttt{UPOQA} Algorithm}\label{alg:UPOQA}
    \KwIn{Objective $f$, elements $\{f_{i}\}_{i=1}^q$, index sets $\{\mathcal{I}_{i}\}_{i=1}^q$, starting point $x_{\text{start}}$, initial trust-region resolution $\rho_0$}
    \KwOut{Approximate solution to problem (\ref{eq:PSProb})}
    Initialize $\{\mathcal{Y}_{0,i}\}_{i=1}^q$ and $\{m_{0,i}\}_{i=1}^q$ centered at $x_{\text{start}}$, set $\boldsymbol{\Delta}_0 = [\rho_0,\ldots,\rho_0]^\top \in \Real^q$\;\label{algstep:UPOQAinit}

    Compute an improved starting point $x_0$ using the method described in Section \ref{sec:upoqa_start}\;\label{algstep:searchx0}
    
    \For{$k = 0, 1, \ldots$}{
        $F^{GI}_1,\ldots, F^{GI}_q \from \false$\;
       
        Solve the trust-region subproblem\begin{eqnarray*}
            &\min\limits_{s\in \Real^n}\ & \sum_{i = 1}^q m_{k,i}(x_k + s),\\
            &\text{s.t.}\ &
            \|s^{\mathcal{I}_i}\|_2 \leq \Delta_{k,i},\quad\text{for } i = 1,2,\ldots, q, \nonumber
        \end{eqnarray*}\label{algstep:UPOQAtrProb}
        to obtain the step $s_k$ using the method detailed in Section \ref{sec:structured_tr} and set $\hat{x}_k \from x_k + s_k$\;

        \eIf{$\max_i\left(\left\|s_k^{\mathcal{I}_i}\right\|_2\right) \leq \rho_k/2$}{
            \For{$i = 1,2,\ldots,q$}{
            \( \Delta_{k,i} \leftarrow \max\left(\Delta_{k,i}/2, \rho_k\right) \)\;
            
            \If{the conditions for geometry improvement of model $m_{k,i}$ are met}{
            $F^{GI}_i \from \true$\;
            }
            }
            \If{entered this branch \textit{in multiple consecutive iterations} $\mathbf{or}$ $F^{GI}_1,,\ldots,F^{GI}_q$ are all \false}{
            $\rho_k \leftarrow 0.1\rho_k$\;
            }
        }{
            Compute\begin{equation*}
                r_k \from \frac{f(\hat{x}_k) - f(x_k)}{m_k(\hat{x}_k) - m_k(x_k)},\quad r_{k,i} \from \frac{f_i(\hat{x}_k) - f_{i}(x_k)}{m_{k,i}(\hat{x}_k) - m_{k,i}(x_k)},\quad \text{for } i = 1,2,\ldots,q;
            \end{equation*}

            \For{$i = 1,2,\ldots,q$}{
                \If{$\hat{x}_k^{\mathcal{I}_i}$ does not significantly deteriorate the poisedness of $\mathcal{Y}_{k,i}$}{\label{algstep:filter_if}
                    Select an interpolation point $y_k^{\text{del},i}$ to remove from $\mathcal{Y}_{k,i}$\;
                    Set $\mathcal{Y}_{k,i} \from \mathcal{Y}_{k,i} \backslash \left\{y_k^{\text{del},i}\right\} \cup {\hat{x}_k^{\mathcal{I}_i}}$ and update the element model $m_{k,i}$\;
                }
                Adjust the trust-region radius $\Delta_{k,i}$ using the method detailed in Section \ref{sec:upoqa_adjust_tr}\;\label{algstep:adjust_delta}
            }
            \If{$r_k \leq 0.1$}{
                \For{$i = 1,2,\ldots,q$}{
                    \If{the conditions for geometry improvement of model $m_{k,i}$ are met}{
            $F^{GI}_i \from \true$;
            }
                }
                \If{$F^{GI}_1,,\ldots,F^{GI}_q$ are all \false}{
                $\rho_k \leftarrow 0.1\rho_k$\;
            }
            }
        }
        Set $x_{k+1}$ as the point with the smaller function value between $x_k$ and $\hat{x}_k$\;
        Perform geometry-improving steps for all element models and interpolation sets where $F^{GI}_i = \true$\;
    }
\end{algorithm}

\subsection{Structured Trust Region}\label{sec:structured_tr}

\texttt{UPOQA} assigns an independent trust-region radius $\Delta_{k,i}$ to each element model $m_{k,i}$. Consequently, the feasible region of the trust-region subproblem is no longer a ball but rather an $n$-dimensional cylinder set defined by the projections $\left\{P_{\mathcal{R}_i}\colon \Real^n \to \mathcal{R}_i\right\}_{i=1}^q$ and the regions $\left\{\mathcal{B}_i(\Delta_{k,i}) \subset \mathcal{R}_i\right\}_{i=1}^q$:
\begin{eqnarray}\label{eq:CylinderDef}
    \mathcal{S}(\boldsymbol{\Delta}_k) &\triangleq&\bigcap_{i=1}^q P_{\mathcal{R}_i}^{-1}\left(\mathcal{B}_i\left(\Delta_{k,i}\right)\right)
    = \left\{s\in \Real^n \mid \|s^{\mathcal{I}_i}\| \leq \Delta_{k,i},\ i = 1,2,\ldots,q \right\},
\end{eqnarray}
where $\boldsymbol{\Delta}_k = [\Delta_{k,1},\ldots,\Delta_{k,q}]$. As before, $\|s^{\mathcal{I}_i}\|$ is shorthand for $\|P_{\mathcal{R}_i} (s)\|$. When the objective takes the form $f(x,y,z) = f_1(x,z) + f_2(y,z)$, $\|s^{\mathcal{I}_i}\|$ uses the $L_2$-norm, and $\Delta_{k,1} = \Delta_{k,2} \triangleq \overline{\Delta}$, the set $\mathcal{S}(\boldsymbol{\Delta}_k)$ will reduce to a Steinmetz solid, which is the intersection of two right circular cylinders with equal radii and orthogonal axes, as illustrated in Figure \ref{fig:standard_steinmetz}.

\begin{figure}
    \centering
    \includegraphics[width=0.5\textwidth]{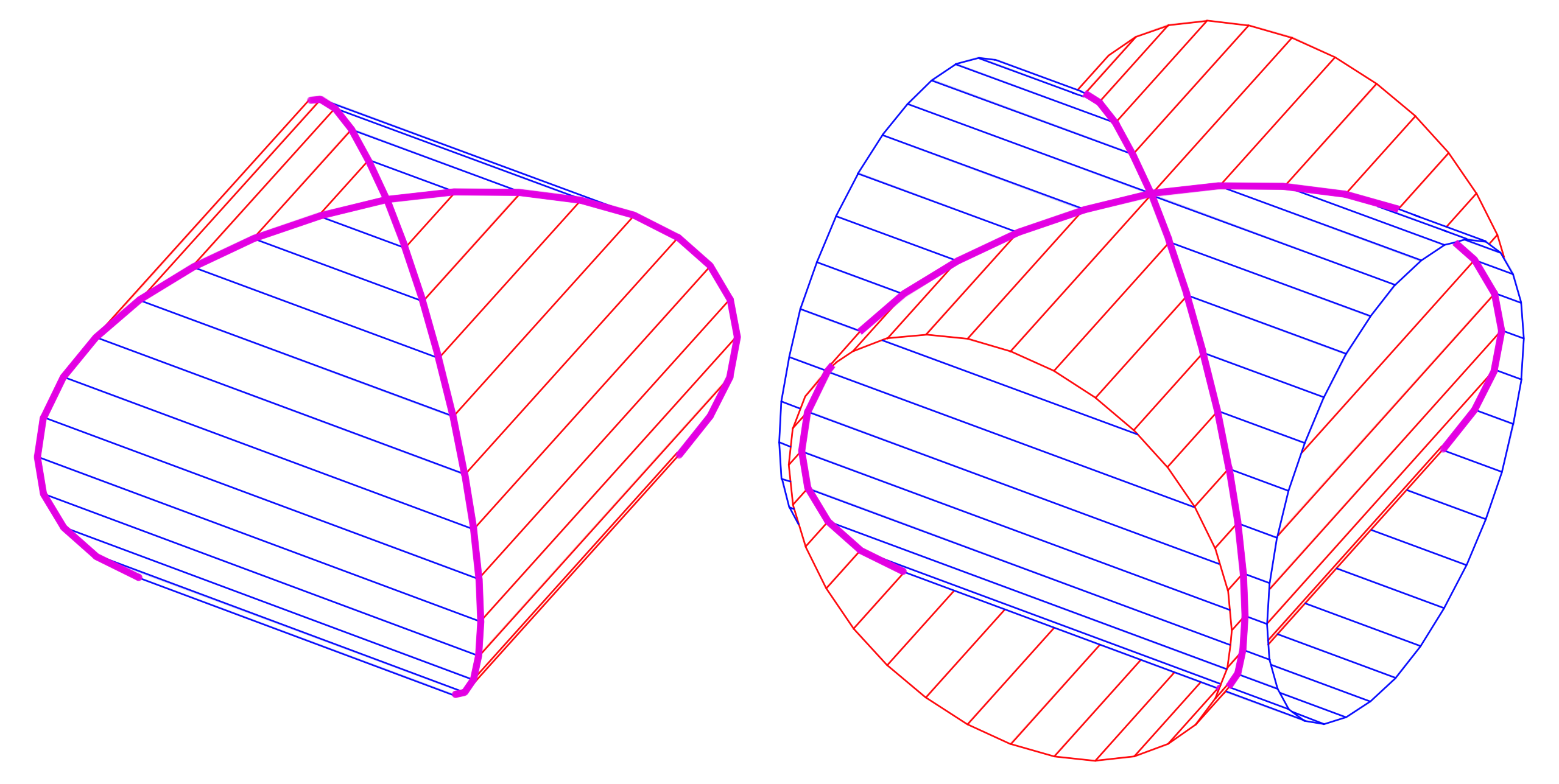}
    \Description{Fully described in the text.}
    \caption[Steinmetz solid]{
        Generation of the surface of a Steinmetz solid.
    }\label{fig:standard_steinmetz}
\end{figure}

When $\Delta_{k,1} \neq \Delta_{k,2}$, the shape of $\mathcal{S}(\boldsymbol{\Delta}_k)$ becomes elongated, indicating that the model is more reliable in one direction than another, thus permitting larger steps in the reliable direction. In more general cases, the geometry of $\mathcal{S}(\boldsymbol{\Delta}_k)$ is highly complex, making the trust-region problem  
\begin{eqnarray}\label{eq:PSTrustRegionProb}
             &\min\limits_{s\in \mathcal{S}(\boldsymbol{\Delta}_k)}\ & m_k(x + s) = \sum_{i = 1}^q m_{k,i}(x + s),
\end{eqnarray}  
rather challenging to solve. It is straightforward to observe that the cylinder set $\mathcal{S}(\boldsymbol{\Delta}_k)$ contains an inscribed ball with radius $\Delta_{k,\min} \triangleq \min\limits_{i=1,\ldots,q}\Delta_{k,i}$, denoted as $\mathcal{B}(\Delta_{k,\min})$. Thus, a cheap alternative is to solve the problem (\ref{eq:PSTrustRegionProb}) approximately within $\mathcal{B}(\Delta_{k,\min})$. To address (\ref{eq:PSTrustRegionProb}), Conn et al.~\cite{PSConnTrustRegion} proposed a two-stage strategy. The first stage involves minimizing $m_k\left(x_k - t \nabla m_k(x_k)\right)$ within $\mathcal{B}(\Delta_{k,\min})$ to obtain a Cauchy step $s_C$. If further reduction can be achieved by increasing the step length, a subsequent minimization of $m_k\left(x_k - t s_C\right)$ is then attempted within $\mathcal{S}(\boldsymbol{\Delta}_k)$. In other DFO methods such as \texttt{PSDFO}~\cite{PSNaiveDFO} and the structured extension of \texttt{BFO}~\cite{PSRandomDirectSearch}, interpolation models are also built for the individual elements, but all element models share a common trust-region radius, thereby circumventing the difficulty of handling structured trust regions. If derivatives are available, it is reasonable to reduce computational costs by simplifying the trust region geometry—for instance, replacing the original region with an inscribed ellipsoid or employing the $L_{\infty}$-norm in the constraints to transform the trust region into a cube~\cite{PSDFOThesis}. Nevertheless, to avoid potential degradation in subproblem solution accuracy caused by such simplifications, we prefer to maintain the trust region \(\mathcal{S}(\boldsymbol{\Delta}_k)\) in its original form. 

Since \(\mathcal{S}(\boldsymbol{\Delta}_k)\) remains convex, employing the projected gradient method is a natural approach for solving problem (\ref{eq:PSTrustRegionProb}). For projecting onto convex sets defined by the intersection of multiple convex sets, \emph{Dykstra’s projection algorithm}~\cite{Dykstra_proj}—which converges precisely to the projection point—is widely adopted. Alternative approaches include the \emph{alternating projection method}~\cite{alt_proj} and \emph{averaged projection method}~\cite{averageProj}, although they both only yield approximate projections. All of these methods achieve linear convergence rates under appropriate conditions~\cite{Dykstra_proj_converge1, Dykstra_proj_converge2, alt_proj_converge, averageProj}. However, the projected gradient method itself often requires numerous iterations to converge. Consequently, solving the trust region subproblem typically constitutes the dominant cost in  the algebraic complexity of the overall algorithm, with the projection operator accounting for a substantial portion of this cost. Moreover, the convergence of these projection methods can require many iterations in certain cases, further increasing the computational burden. 

In order to improve computational efficiency, we introduce a specialized approximate projection operator tailored for trust regions of the form (\ref{eq:CylinderDef}). This operator, termed the \textit{Steinmetz projection}, is designed to converge within a fixed number of iterations. The detailed procedure is presented in Algorithm \ref{alg:PSProj}, which utilizes the subroutine \(\texttt{SHRINK}\) (Algorithm \ref{alg:Shrink}).  For clarity, we define the \emph{violation ratio} of a point \( s \in \mathbb{R}^n \) with respect to the \(i\)-th element as  
\begin{equation*}  
    v_i(s) \triangleq \frac{\left\|s^{\mathcal{I}_i}\right\|_2}{\Delta_i}.  
\end{equation*}  
The core idea of Steinmetz projection stems from an observation: Given an index set $\mathcal{G}$ and a point $s\in \mathbb{R}^n$ where each projection $s^{\mathcal{I}_i}$ satisfies $\|s^{\mathcal{I}_i}\| = \mu\Delta_i$ for all $i\in\mathcal{G}$, then scaling $s$ by any real factor $t$ yields $\|(ts)^{\mathcal{I}_i}\| = t\mu\Delta_i$ for all $i\in\mathcal{G}$. Thus, when projecting \( s \) onto $\mathcal{S}(\boldsymbol{\Delta})$, we first collect the indices of all elements with the highest violation ratios into a set \(\mathcal{G}\), which by construction satisfies \(\|s^{\mathcal{I}_i}\| = \mu\Delta_i\) for some \(\mu \in \mathbb{R}\) and all \(i \in \mathcal{G}\). Next, we rescale \( s \) in the subspaces corresponding to \(\mathcal{G}\) until another element can be added into \(\mathcal{G}\), repeating the process iteratively. Here, \(\mathcal{G}\) is termed the \emph{shrinking set}. Once \(\mathcal{G}\) encompasses all elements and the current point \( s \) satisfies \(\|s^{\mathcal{I}_i}\|_2 = \mu \Delta_i\) for all \( i \), a final scaling by \( 1/\mu \) ensures the resulting point lies within \(\mathcal{S}(\boldsymbol{\Delta})\). As shown in Algorithm \ref{alg:PSProj}, each iteration first identifies the shrinking set and then performs the shrinking operation.

\begin{algorithm}
    \caption{Steinmetz Projection}\label{alg:PSProj}
    \KwIn{Point \( s_0 \in \mathbb{R}^n \), index sets \( \{\mathcal{I}_{i}\}_{i=1}^q \), trust-region radii \( \{\Delta_{i}\}_{i=1}^q \)}
    \KwOut{An approximation of \( P_{\mathcal{S}(\boldsymbol{\Delta})}(s_0) \)}

    \For{\( k = 0, 1, \ldots \)}{
        Compute \( v_{i} \gets \left\|s_0^{\mathcal{I}_i} \right\|_2 / \Delta_{i} \) for \( i = 1,\ldots,q \)\;\label{algstep:PSproj_vi}

        \If{\( \max\limits_{i=1,\ldots,q} v_{k,i} \leq 1 \)}{
            \Break
        }

        \( \mathcal{G}_k \gets \left\{i = 1,\ldots, q \mid v_{k,i} = \max\limits_{i=1,\ldots,q} v_{k,i} \right\} \)\;\label{algstep:PSproj_Gi}
        
        \( s_k \gets \texttt{SHRINK}\left(s_k,~\{\mathcal{I}_{i}\}_{i=1}^q,~\{\Delta_{i}\}_{i=1}^q,~\mathcal{G}_k\right) \)\;
    }
\end{algorithm}

\begin{algorithm}
    \caption{Subroutine $\texttt{SHRINK}$}\label{alg:Shrink}
    \KwIn{Point $s \in \mathbb{R}^n$, index sets $\{\mathcal{I}_{i}\}_{i=1}^q$, trust-region radii $\{\Delta_{i}\}_{i=1}^q$, shrinking set $\mathcal{G}$}
    \KwOut{Shrunk point $s_*$}
    
    Select arbitrary $i\in\mathcal{G}$ and compute $u \gets \left\|s^{\mathcal{I}_i}\right\|_2 / \Delta_i$\;\label{step:trial_i}
    
    \eIf{$\mathcal{G} = \{1,2,\ldots,q\}$}{
        $s_* \gets s/u$\;
    }{
        $\mathcal{I}_{\mathcal{G}} \gets \bigcup_{i\in\mathcal{G}} \mathcal{I}_i$\;
        \For{$i = 1, \ldots,q$}{
            \eIf{$i \in \mathcal{G}$}{
                $t_i \gets 1/u$\;\label{algstep:ti_calc_1}
            }{
                $\displaystyle t_i \gets \max\left\{\frac{1}{u},\ \frac{\|s^{ \mathcal{I}_i \backslash \mathcal{I}_{\mathcal{G}}} \|_2}{\sqrt{ u^2\Delta_i^2 - \|s^{ \mathcal{I}_i \cap \mathcal{I}_{\mathcal{G}}} \|_2^2 }}\right\} $\;\label{algstep:ti_calc_2}
            }
        }
        $t_* \from \max_{i = 1,\ldots,q} t_i$\;\label{algstep:ti_calc_max}
        $s_* \from s +  \left(t_* - 1\right)s^{\mathcal{I}_{\mathcal{G}}} $\;
    }
\end{algorithm}

The $\texttt{SHRINK}$ subroutine (Algorithm \ref{alg:Shrink}) implements the shrinking operation. Assuming without loss of generality that the given point $s$ lies outside $\mathcal{S}(\boldsymbol{\Delta})$, and given a shrinking set $\mathcal{G} = \{1,2,\ldots,\hat{q}\}\ (\hat{q} < q)$ satisfying  
\begin{eqnarray}\label{eq:shrink_group_start_prop}  
    v_1(s) = \cdots = v_{\hat{q}}(s) > v_{i}(s),\quad \forall i \in\{\hat{q} + 1, \ldots,q\},
\end{eqnarray}  
the subroutine performs component-wise contraction in the subspace $\mathcal{R}_{\mathcal{G}} \triangleq \operatorname{span}\left(\cup_{i\in\mathcal{G}} \mathcal{R}_i\right)$. Specifically, it scales the component $s^{\mathcal{I}_{\mathcal{G}}}$ by a factor $t_* \in \mathbb{R}$, generating a new point $s_* \triangleq s + (t_*-1) s^{\mathcal{I}_{\mathcal{G}}}$ that satisfies either:  
\begin{eqnarray}\label{eq:shrink_group_end_prop1}  
    1 = v_1(s_*) = \cdots = v_{\hat{q}}(s_*) \geq v_{i}(s_*),\quad \forall i \in\{\hat{q} + 1, \ldots, q\},  
\end{eqnarray}  
or introduces an elemental index $\hat{p}\in \{\hat{q} + 1, \ldots, q\}$ such that  
\begin{eqnarray}\label{eq:shrink_group_end_prop2}  
    v_1(s_*) = \cdots = v_{\hat{q}}(s_*) = v_{\hat{p}}(s_*) \geq v_{i}(s_*),\quad \forall i \in\{\hat{q} + 1, \ldots, q\}\setminus \{\hat{p}\}, 
\end{eqnarray}  
where $\mathcal{I}_{\mathcal{G}} \triangleq \bigcup_{i\in\mathcal{G}} \mathcal{I}_i.$

This mechanism enables the outer loop (Algorithm \ref{alg:PSProj}) to terminate the projection or subsequently expand the shrinking set to $\mathcal{G}_{\text{new}} = \{1,2,\ldots,\hat{q}\}\cup\{\hat{p}\}$. The elemental scaling ratio $t_i$ is the maximum of two possible choices, as computed in Steps \ref{algstep:ti_calc_1} and \ref{algstep:ti_calc_2} of Algorithm \ref{alg:Shrink}, yielding the candidate point $s_i \triangleq s + (t_i-1) s^{\mathcal{I}_{\mathcal{G}}}$ satisfying either $1 = v_1(s_i) = \cdots = v_{\hat{q}}(s_i) > v_{i}(s_i)$ or $v_1(s_i) = \cdots = v_{\hat{q}}(s_i) = v_{i}(s_i)$.  The final contraction ratio $t_*$ is then selected as the maximum value of $t_1,\ldots,t_q$ ensuring satisfaction of (\ref{eq:shrink_group_end_prop1}) or (\ref{eq:shrink_group_end_prop2}). Note that the selection of trial index $i$ in Step \ref{step:trial_i} is arbitrary, as the ratio $\left\|s^{\mathcal{I}_i}\right\|_2 / \Delta_i$ remains constant across all $i\in\mathcal{G}$ due to the defining property of the shrinking set $\mathcal{G}$. Theorem \ref{theo:converg_shrink} formally establishes the correctness of the $\texttt{SHRINK}$ subroutine.

\begin{theorem}\label{theo:converg_shrink}
    Given a point $s \notin \mathcal{S}(\boldsymbol{\Delta})$, element index sets $\{\mathcal{I}_{i}\}_{i=1}^q$, trust-region radii $\{\Delta_{i}\}_{i=1}^q$, and a shrinking set $\mathcal{G}$ satisfying condition (\ref{eq:shrink_group_start_prop}), there exists $\hat{p}\in \{\hat{q} + 1, \ldots, q\}$ such that the point $s_*$ returned by Algorithm \ref{alg:Shrink} will satisfy either (\ref{eq:shrink_group_end_prop1}) or (\ref{eq:shrink_group_end_prop2}).
\end{theorem}
\begin{proof}
Let $\hat{p} = \operatorname{argmax}_{i\notin \mathcal{G}} t_i$, where $t_i$ is determined by Steps \ref{algstep:ti_calc_1} and \ref{algstep:ti_calc_2} of Algorithm \ref{alg:Shrink}. Let $u = \left\|s^{\mathcal{I}_i}\right\|_2 / \Delta_i$ for an arbitrary $i\in\mathcal{G}$, noting that this value is independent of the choice of $i$. Define $t_* = \max_{i = 1,\ldots,q} t_i$ and $s_* = s + (t_* - 1)s^{\mathcal{I}_{\mathcal{G}}}$. We consider the size of $u^2\Delta_{\hat{p}}^2$ in two cases. 

First, if $u^2 \Delta_{\hat{p}}^2 > u^2 \left\|s^{ \mathcal{I}_{\hat{p}}\backslash \mathcal{I}_{\mathcal{G}}} \right\|_2^2 + \left\|s^{ \mathcal{I}_{\hat{p}} \cap \mathcal{I}_{\mathcal{G}}}\right\|_2^2$ holds, then we have 
\begin{eqnarray*}
    \frac{1}{u} > \frac{\|s^{ \mathcal{I}_i \backslash \mathcal{I}_{\mathcal{G}}} \|_2}{\sqrt{ u^2\Delta_i^2 - \|s^{ \mathcal{I}_i \cap \mathcal{I}_{\mathcal{G}}} \|_2^2 }},
\end{eqnarray*}
and thus $t_{\hat{p}} = 1/u$. On one hand, $t_i = 1/u ~ (\forall i \in \mathcal{G})$, and on the other hand, $t_i \leq t_{\hat{q}} = 1/u~ (\forall i \notin \mathcal{G})$. Thus, $t_* = 1/u$. For any $i \in \mathcal{G}$, since $\mathcal{I}_i \subset \mathcal{I}_{\mathcal{G}}$, we have
\begin{equation}\label{eq:converg_shrink_step1}
    s_*^{\mathcal{I}_i} = P_{\mathcal{R}_i} s_* = t_* s^{\mathcal{I}_i},\quad\forall i \in \mathcal{G}.
\end{equation}
Furthermore, by the definition of $u$, it follows that
\begin{align*}
    s_*^{\mathcal{I}_i} =  \frac{1}{u} s^{\mathcal{I}_i}
    \Longrightarrow \left\| s_*^{\mathcal{I}_i}\right\|_2 = \frac{1}{u} \left\|s^{\mathcal{I}_i}\right\|_2 = \Delta_i
    \Longrightarrow v_i(s_*) = 1
    ,\quad&\forall i \in \mathcal{G}.
\end{align*}
Meanwhile, for any $i \notin \mathcal{G}$, since  
\begin{eqnarray}\label{eq:converg_shrink_step2}  
    s_*^{\mathcal{I}_i} =  
    s^{\mathcal{I}_i \backslash \mathcal{I}_{\mathcal{G}}} +   
    t_* s^{\mathcal{I}_i \cap \mathcal{I}_{\mathcal{G}}},\quad\forall i \notin \mathcal{G},  
\end{eqnarray}  
and $s^{\mathcal{I}_i \backslash \mathcal{I}_{\mathcal{G}}} \perp  s^{\mathcal{I}_i \cap \mathcal{I}_{\mathcal{G}}}$, we have  
\begin{align*}  
    &\left\|s_*^{\mathcal{I}_i}\right\|_2^2 =  
    \left\|s^{\mathcal{I}_i \backslash \mathcal{I}_{\mathcal{G}}}\right\|_2^2 + \frac{1}{u^2} \left\|s^{\mathcal{I}_i \cap \mathcal{I}_{\mathcal{G}}}\right\|_2^2   
 \leq \Delta_i^2 \Longrightarrow  
 v_i(s_*) \leq 1,\quad\forall i \notin \mathcal{G}.
\end{align*}  
Thus, in this case, inequality (\ref{eq:shrink_group_end_prop1}) holds.

In the second case, when \( u^2 \Delta_{\hat{p}}^2 \leq u^2 \left\|s^{ \mathcal{I}_{\hat{p}} \backslash \mathcal{I}_{\mathcal{G}}} \right\|_2^2 + \left\|s^{ \mathcal{I}_{\hat{p}} \cap \mathcal{I}_{\mathcal{G}}}\right\|_2^2 \) holds, we have \( t_{\hat{p}} \geq 1/u \), and thus \( t_* = t_{\hat{p}} \). For any \( i \in \mathcal{G} \), since equation (\ref{eq:converg_shrink_step1}) still holds, it follows that \( v_i(s_*) = t_* v_i(s) = t_{\hat{p}} u \). Therefore, all \( v_i(s_*) \) (\( \forall i \in \mathcal{G} \)) remain the same. From equation (\ref{eq:converg_shrink_step2}), we derive  
\begin{eqnarray*}  
    \left\|s_*^{\mathcal{I}_{\hat{p}}}\right\|_2^2 = 
    \left\|s^{\mathcal{I}_{\hat{p}} \backslash \mathcal{I}_{\mathcal{G}}}\right\|_2^2 + t_{\hat{p}}^2 \left\|s^{\mathcal{I}_{\hat{p}} \cap \mathcal{I}_{\mathcal{G}}}\right\|_2^2
    =
    \left\|s^{\mathcal{I}_{\hat{p}} \backslash \mathcal{I}_{\mathcal{G}}}\right\|_2^2 +   
   \frac{\left\|s^{ \mathcal{I}_{\hat{p}} \backslash \mathcal{I}_{\mathcal{G}}} \right\|_2^2\left\|s^{\mathcal{I}_{\hat{p}} \cap \mathcal{I}_{\mathcal{G}}}\right\|_2^2}{  
               u^2\Delta_{\hat{p}}^2 - \left\|s^{ \mathcal{I}_{\hat{p}} \cap \mathcal{I}_{\mathcal{G}}} \right\|_2^2  
            }   
         =
         \frac{u^2\Delta_{\hat{p}}^2\left\|s^{\mathcal{I}_{\hat{p}}\backslash \mathcal{I}_{\mathcal{G}}}\right\|_2^2  }{  
               u^2\Delta_{\hat{p}}^2 - \left\|s^{ \mathcal{I}_{\hat{p}} \cap \mathcal{I}_{\mathcal{G}}} \right\|_2^2  
            }  
            =  
          u^2  \Delta_{\hat{p}}^2 t_{\hat{p}}^2. \nonumber  
\end{eqnarray*}  
Thus, we have \( v_{\hat{p}}(s_*) = t_{\hat{p}} u \). For any \( i \notin \mathcal{G} \), we claim that \( v_i(s_*) \leq t_{\hat{p}} u \). Otherwise, the inequality \( \left\|s_*^{\mathcal{I}_i}\right\|_2 > t_{\hat{p}} u \Delta_i \) would hold, leading to  
\begin{eqnarray*}  
   \left\|s^{\mathcal{I}_i \backslash \mathcal{I}_{\mathcal{G}}}\right\|_2^2 + t_{\hat{p}}^2 \left\|s^{\mathcal{I}_i \cap \mathcal{I}_{\mathcal{G}}}\right\|_2^2 > t_{\hat{p}}^2 u^2 \Delta_i^2 \Longrightarrow
   \frac{\left\|s^{ \mathcal{I}_i \backslash \mathcal{I}_{\mathcal{G}}} \right\|_2}{\sqrt{  
               u^2\Delta_i^2 - \left\|s^{ \mathcal{I}_i \cap \mathcal{I}_{\mathcal{G}}} \right\|_2^2  
            }} > t_{\hat{p}},  
\end{eqnarray*}  
which implies \( t_i > t_{\hat{p}} \). This contradicts the assumption \( \hat{p} = \operatorname{argmax}_{i\notin \mathcal{G}} t_i \). Therefore,  we have
\begin{equation*}  
    t_{\hat{p}} u = v_j(s_*)  = v_{\hat{p}}(s_*) \geq v_{i}(s_*), \quad \forall ~i \in \mathcal{G}^c\backslash \{\hat{p}\},~ j\in\mathcal{G}.  
\end{equation*}  
Hence, inequality (\ref{eq:shrink_group_end_prop2}) holds.  
\end{proof}


Note that at the initialization of Steinmetz projection, we may exclude the index sets $\mathcal{I}_j$ of elements satisfying $v_j(s_0) \leq 1$ from $\{\mathcal{I}_{i}\}_{i=1}^q$, since they never contribute $t_j$ values exceeding $1/u$ and thus cannot affect the value of $t_*$. Furthermore, Steps \ref{algstep:PSproj_vi} and \ref{algstep:PSproj_Gi} of Algorithm \ref{alg:PSProj} need not be explicitly recomputed in each iteration, as the new shrinking set can always be obtained simply using intermediate results from the subroutine $\texttt{SHRINK}$. If $\left|\left\{i\mid v_j(s_0) \leq 1\right\}\right| = \overline{q}$, then each run of Algorithm \ref{alg:Shrink} has a cost of $O\left(n\left(q - \overline{q} - |\mathcal{G}| \right)\right)$. Since each run either terminates the projection or adds at least one new member to the shrinking set, the algorithm iterates at most $q - \overline{q}$ times. Therefore, the computational complexity of Algorithm \ref{alg:PSProj} is $O\left(
    \sum_{i=1}^{q-\overline{q}} n(q-\overline{q} - i)
    \right) 
    =
     O\left(
    n(q-\overline{q})^2
    \right)$.

Note that the Steinmetz projection is only an approximate method with moderate accuracy. To compensate, we adopt a hybrid strategy: first apply the averaged projection method for \(k_{\text{avg}} = 1\) to \(5\) iterations, then project the resulting point using Steinmetz projection. For any point \(s \in \mathbb{R}^n\) to be projected, we denote the output of this combined approach as \(\widehat{P}_{\mathcal{S}(\boldsymbol{\Delta})}\left(s;~k_{\text{avg}}\right)\). The more accurate Dykstra’s algorithm cannot be used as a precursor here, since it also iterates on auxiliary vectors and its accuracy depends heavily on these vectors, not solely on the iterates. Consequently, stopping early to obtain an intermediate point from Dykstra’s algorithm is meaningless in this context.

To solve the trust-region subproblem (\ref{eq:PSTrustRegionProb}), we employ a modified projected gradient method (Algorithm \ref{alg:modified_pg}), which closely follows the truncated conjugate gradient method. The key difference arises when the iterate $s_k + \alpha_k p_k$ falls outside $\mathcal{S}(\boldsymbol{\Delta})$. Instead of stopping, the algorithm performs an exact line search along the direction $\widehat{P}_{\mathcal{S}(\boldsymbol{\Delta})}(s_k + \alpha_k p_k;~k_{\text{avg}}) - s_k$ starting from $s_k$ to determine $s_{k+1}$. Furthermore, the algorithm alternates between $k_{\text{avg}}^{(1)} = 0$ and $k_{\text{avg}}^{(2)} = 4$ during iterations to mitigate the limitations of using a single approximate projection operator.
The truncation radius \(\widehat{\Delta}\) in Algorithm \ref{alg:modified_pg} is set to $\sqrt{\min\{n,q\}} \cdot \max_{1\leq i \leq q}\{\Delta_{k,i}\}$, which ensures that the trust region is entirely contained within the ball \(\mathcal{B}(\widehat{\Delta})\). In addition, a comparison of \texttt{UPOQA}'s performance with versus without structured trust regions is provided in Appendix \ref{sec:non_struct}, where the non-structured version employs a spherical trust region, and the resulting subproblem is solved via the truncated conjugate gradient method with a boundary improvement step.

\begin{algorithm}
    \caption{Modified Projected Gradient Method for the Trust-Region Subproblem (\ref{eq:PSTrustRegionProb})}\label{alg:modified_pg}
    \KwIn{Overall model $m$, center point $x$, truncation radius $\widehat{\Delta} > 0$, projection operator $\widehat{P}_{\mathcal{S}(\boldsymbol{\Delta})}$, numbers of start-up averaged projection iterations $k_{\text{avg}}^{(1)},~k_{\text{avg}}^{(2)} \geq 0$}
    \KwOut{An approximate solution to (\ref{eq:PSTrustRegionProb})}
    Set the iterate $s_0 \from 0$\;
    Set the search direction $p_0 \from - \nabla m(x + s_0)$\;
    \For{$k = 0, 1, \ldots~\mathbf{until}~\nabla m(x+s_k)^\top  p_k \geq 0$ }{

        Compute $\alpha_k \from \argmin \left\{m(x + s_k + \alpha p_k)\mid \alpha \geq 0,~ \| s_k + \alpha p_k \|\leq \widehat{\Delta}\right\}$\;

        \eIf{$s_k + \alpha^k p_k$ falls outside $\mathcal{S}(\boldsymbol{\Delta})$}{
            $\hat{s}_{k} \from \widehat{P}_{\mathcal{S}(\boldsymbol{\Delta})}\left(s_k + \alpha_k p_k;~k_{\text{avg}}^{(k~\mathrm{mod}~2)}\right)$\;
            $\hat{p}_k\from \hat{s}_k - s_k$\;
            $\hat{\alpha}_k \from \argmin \left\{m(x + s_k + \alpha \hat{p}_k)\mid 0 \leq \alpha \leq 1\right\}$\;
            $s_{k+1} \from s_k + \hat{\alpha}_k \hat{p}_k$\;
            $p_{k+1} \from - \nabla m(x + s_{k+1})$\;
        }{
            $s_{k+1} \from s_k + \alpha_k p_k$\;
            $\beta_{k} \from \left\|\nabla m(x + s_{k+1})\right\|^2 / \left\|\nabla m(x+s_k)\right\|^2$\;
            $p_{k+1} \from- \nabla m(x + s_{k+1}) + \beta_k p_k $\;
        }
    
    }
    
\end{algorithm}

\subsection{Managment of Trust Region Radii}\label{sec:upoqa_adjust_tr}

In Step \ref{algstep:adjust_delta} of Algorithm \ref{alg:UPOQA}, we employ the combined separation criterion proposed by Shahabuddin~\cite[][Section 2.4]{PSDFOThesis} to adjust the trust region radii $\{\Delta_{k,i}\}_{i=1}^q$. Here we briefly introduce this criterion.

In trust region methods, it is common practice to adjust the trust region based on the value of the reduction ratio $r_k = \delta_k f / \delta_k m_k$, where
\begin{equation*}
    \delta_k f = f(x_k) - f(\hat{x}_k),\quad  
    \delta_k m_{k} = m_{k}(x_{k}) - m_{k}(\hat{x}_k).
\end{equation*}
However, when solving partially separable problems, it is infeasible to independently adjust each $\Delta_{k,i}$ solely based on $r_{k,i}$. For an individual objective, we always expect the surrogate model and the objective function to decrease. However, for coexisting element functions, the potential antagonistic relationship between them may necessitate that we sometimes tolerate a slight deterioration in the function values of some elements. This suggests that the algorithm should not simply encourage larger $r_{k,i}$ values. Moreover, the cancellation issue further complicates the problem. Conn et al.~\cite{PSConnTrustRegion} first pointed out that a situation may arise where $\delta_k m_{k,1}, \ldots, \delta_k m_{k,q}$ have large absolute magnitudes, but cancellation occurs, resulting in a very small absolute value for the aggregated $\delta_k m_k$. In such cases, rounding errors in $\delta_k m_{k,1}, \ldots, \delta_k m_{k,q}$ can easily lead to inaccurate computation of $\delta_k m_k$ and to misjudgment of the reliability of the overall model $m_k$.

For above considerations, Shahabuddin~\cite[][Section 2.4]{PSDFOThesis} proposed a scoring-based strategy. The scores consist of a global score \(\tau_k\) and element-wise local scores \(\{\hat{\tau}_{k,i}\}_{i=1}^q\), both taking integer values of 0, 1, or 2. By summing the two parts, the total score (ranging from 0 to 4) determines how each element's trust region radius \(\Delta_{k,i}\) should be adjusted. The global score \(\tau_k\) depends solely on the magnitude of \(r_k\)—the larger \(r_k\), the higher \(\tau_k\). The element-wise score \(\hat{\tau}_{k,i}\) is determined by the position of the point \((\delta_k m_{k,i}, \delta_k f_{i})\). Figure \ref{fig:separation_criterion} illustrates the acceptable region for assigning \(\{\hat{\tau}_{k,i}\}_{i=1}^q\), bounded by two rays with different slopes emanating from the origin and a line with slope 1 that does not pass through the origin. The stricter the acceptable region in which \((\delta_k m_{k,i}, \delta_k f_{i})\) falls, the higher \(\hat{\tau}_{k,i}\). The detailed scoring strategy is described in Algorithm \ref{alg:tr_score}. After scoring, we set $\tau_{k,i} = \tau_k + \hat{\tau}_{k,i}$. The new trust region radii are then updated as follows:

\begin{equation*}
    \Delta_{k+1,i}
    \begin{cases}
        \in [1,~\theta_4] \Delta_{k,i}, & \text{if }~\tau_{k,i} = 4, \\
        \in [1,~\theta_3] \Delta_{k,i}, & \text{if }~\tau_{k,i} = 3, \\
        \in [\theta_2,~1] \Delta_{k,i}, & \text{if }~\tau_{k,i} = 2, \\
        =\theta_2 \Delta_{k,i}, & \text{if }~ \tau_{k,i} = 1, \\
        =\theta_1 \Delta_{k,i}, & \text{if }~ \tau_{k,i} = 0,
    \end{cases}
\end{equation*}

In our implementation, the parameters are set as \(\theta_1 = 0.5\), \(\theta_2 = 1 / \sqrt{2}\), \(\theta_3 = \sqrt{2}\), and \(\theta_4 = 2\). Additionally, Step \ref{algstep:prevent_enlarge} in Algorithm \ref{alg:tr_score} ensures that when the overall model reduction is unsatisfactory (\(r_k < \mu_1\)), at least one element's trust region radius undergoes substantial decrease to prevent the algorithm from stagnating at a fixed radius. Although such cases are rare, they have indeed been observed in our numerical experiments. This step is the only modification we made to Shahabuddin's combined separation criterion~\cite[][Section 2.4]{PSDFOThesis}.

\begin{figure}
    \centering
    \includegraphics[width = 0.55\textwidth]{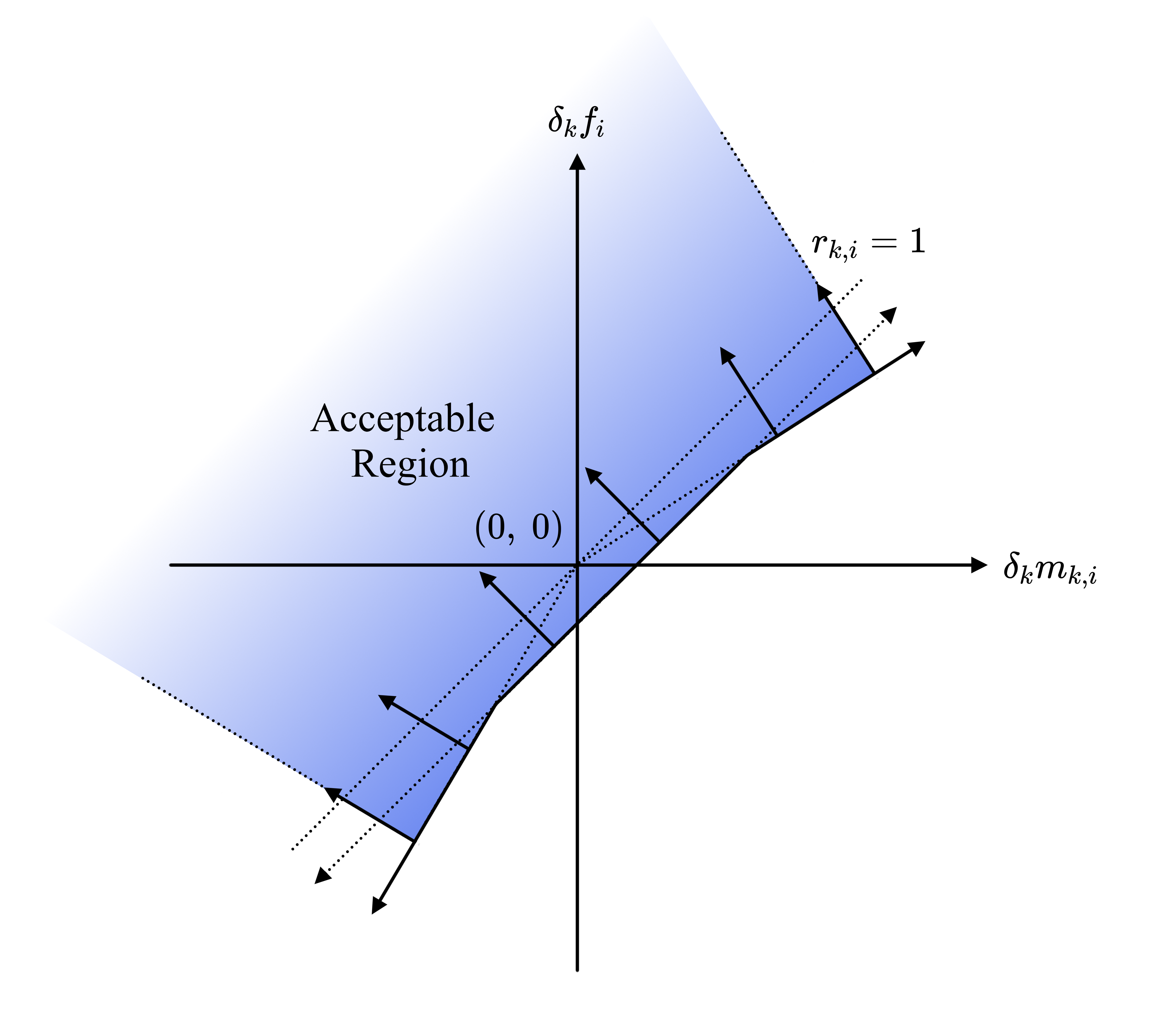}
    \caption{The combined separation criterion proposed by Shahabuddin~\cite[][Section 2.4]{PSDFOThesis} determines whether to expand the elemental trust region radius \(\Delta_{k,i}\) based on the slope \(r_{k,i}\) of the point \((\delta_k m_{k,i},~\delta_k f_{i})\) relative to the origin and its distance to the line \(r_{k,i} = 1\). The colored region in the figure represents the acceptable region for expansion operations.}
    \Description{Fully described in the text.}
    \label{fig:separation_criterion}
\end{figure}

\begin{algorithm}
    \caption{Scoring Strategy for Trust Region Radius Adjustment}\label{alg:tr_score}
    \KwIn{Elemental model reductions $\{\delta_k m_{k,i}\}_{i=1}^q$, elemental reduction ratios $\{r_{k,i}\}_{i=1}^q$, trust region radii $\{\Delta_{k,i}\}_{i=1}^q$, trust region resolution $\rho_k$, overall reduction ratio $r_k$, constants $0 < \mu_1 < \mu_2 < 1$}
    \KwOut{Total scores for each element $\{\tau_{k,i}\}_{i=1}^q$}

    $\displaystyle \zeta_k \gets \frac{\sum_{\delta_k m_{k,i} < 0} \delta_k m_{k,i}}{\sum_{\delta_k m_{k,i} \geq 0} \delta_k m_{k,i}}$\;

    For $i=1,2$, set $\eta_i= - (1 - \mu_i) \zeta_k$, $\displaystyle \alpha_i \gets \frac{(\mu_i + \eta_i)(1+\zeta_k) - 2\zeta_k}{1-\zeta_k}$\;

    \lIf{$r_k \geq \mu_2$}{
    $\tau_k\gets 2$
    }\lElseIf{$r_k \geq \mu_1$}{
    $\tau_k\gets 1$
    }\lElse{
    $\tau_k\gets 0$
    }

    \For{$i = 1,2,\ldots,q$}{
        \eIf{$\delta_k m_{k,i} \geq 0$}{
            \uIf{$r_{k,i} \geq \alpha_2~\mathbf{or}~\delta_k f_{i} \geq \delta_k m_{k,i} - \eta_2 (\delta_k m_{k}) / q$ }{
                $\hat{\tau}_{k,i}\gets 2$\;
            }\uElseIf{$r_{k,i} \geq \alpha_1~\mathbf{or}~\delta_k f_{i} \geq \delta_k m_{k,i} - \eta_1 (\delta_k m_{k}) / q$}{
                $\hat{\tau}_{k,i}\gets 1$\;
            }
            \lElse{
                $\hat{\tau}_{k,i}\gets 0$
            }
        }{
            \uIf{$r_{k,i} \leq 2 - \alpha_2~\mathbf{or}~\delta_k f_{i} \geq \delta_k m_{k,i} - \eta_2 (\delta_k m_{k}) / q$}{
                $\hat{\tau}_{k,i}\gets 2$\;
            }\uElseIf{$r_{k,i} \leq 2 - \alpha_1~\mathbf{or}~\delta_k f_{i} \geq \delta_k m_{k,i} - \eta_1 (\delta_k m_{k}) / q$}{
                $\hat{\tau}_{k,i}\gets 1$\;
            }\lElse{
                $\hat{\tau}_{k,i}\gets 0$
            }
        }
         $\tau_{k,i}\gets \hat{\tau}_{k,i} + \tau_{k}$\;
    }
    \If{$\tau_k = 0$}{
        Verify all $\tau_{k,i}$ to ensure at least one element satisfies $\tau_{k,i} \leq 1$ and $\Delta_{k,i} >\rho_k$. Otherwise, set $\tau_{k,i}$ of the lowest-scoring element with $\Delta_{k,i} > \rho_k$ to $0$\; \label{algstep:prevent_enlarge}
    }
\end{algorithm}\par

\subsection{Selective Updates of Element Models}

After obtaining the solution $s_k$ to the trust-region subproblem (\ref{eq:PSTrustRegionProb}), it is time to update element models and interpolation sets. For the model $m_{k,i}$, the new interpolation point generated by the solution $s_k$ is $\hat{x}_k^{\mathcal{I}_i} = (x_k+s_k)^{\mathcal{I}_i}$. However, blindly adding $\hat{x}_k^{\mathcal{I}_i}$ to the set $\mathcal{Y}_{k,i}$ may deteriorate the poisedness of $\mathcal{Y}_{k,i}$. For example, when the objective function is $f(x,y) = f_1(x) + f_2(y)$ and the starting point is $(x_0,y_0)$, if $x_0$ is already very close to a local minimizer of $f_1$, the computed trust-region step will have vastly different scales in the subspaces corresponding to $x$ and $y$. This can cause a drastic contraction in the geometry of the interpolation set $\mathcal{Y}_1$, potentially leading to numerical instability. If the algorithm detects such a risk, it skips such updates, as implemented in Step \ref{algstep:filter_if} of Algorithm \ref{alg:UPOQA}.

In Powell's DFO methods that based on quadratic models, the denominator $\sigma$ in the update formula for the interpolation system is a crucial parameter~\cite{PowellUpdateInverseKKT}. Its magnitude largely determines the impact of the update on the poisedness of the interpolation system, and a larger value is preferred. Therefore, we decide whether to accept the new interpolation point $\hat{x}_k^{\mathcal{I}_i}$ into $\mathcal{Y}_{k,i}$ by checking the value $\sigma_{k,i}$ during the update of $m_{k,i}$. We first compute $\sigma_{k,i}$ and multiply it by a factor $\gamma_{k,i}\triangleq \min \left\{\left\|s_k^{\mathcal{I}_i}\right\|_2/\rho_{k},~1\right\}$, where $s_k$ is the solution to the trust-region subproblem (\ref{eq:PSTrustRegionProb}) solved in Step \ref{algstep:UPOQAtrProb} of Algorithm \ref{alg:UPOQA}. The factor $\gamma_{k,i}$ penalizes situations with small step updates, as they generally contribute limited improvements to the model's accuracy, even if they don't significantly compromise the system's poisedness. Due to rounding errors, the value of $\sigma_{k,i}$ may become negative. To mitigate this, we check the intermediate variables used in computing $\sigma_{k,i}$ to detect such tendencies, and apply an additional penalty factor $w_{k,i}$ to the value of $\sigma_{k,i}$. This penalty factor is determined heuristically - in short, we reduce positive $\sigma_{k,i}$ values while enlarging negative ones. The resulting value $w_{k,i}\gamma_{k,i}\sigma_{k,i}$ is then compared with a fixed threshold $\xi$. The new interpolation point $\hat{x}_k^{\mathcal{I}_i}$ is accepted only if $w_{k,i}\gamma_{k,i}\sigma_{k,i}$ exceeds the threshold. In \texttt{UPOQA}, the default threshold is set to $\xi = 10^{-5}$. 

In most cases, updates with negative $\sigma_{k,i}$ are automatically rejected. However, an exception occurs in extreme situations where all candidate updates yield negative $\sigma_{k,i}$ values, and the algorithm will accept the update whose $w_{k,i}\gamma_{k,i}\sigma_{k,i}$ is closest to zero. Only then does the additional penalty on negative $\sigma_{k,i}$ values matter. 

\subsection{Starting Point Search}\label{sec:upoqa_start}

As mentioned earlier, based on the element values evaluated at a certain number of interpolation points selected from the grid  
\begin{equation*}
    \left\{x = \sum_{i = 1}^n c_i e_i \mid c_1,\ldots,c_n \in \mathbb{Z}\right\},
\end{equation*}
one can easily compute significantly more overall function values. This idea has been adopted by Price and Toint~\cite{PSMeshPatternSearch}. Although \texttt{UPOQA} cannot directly utilize this search pattern to accelerate the algorithm, it employs similar strategy to select a better starting point for the algorithm.

\texttt{UPOQA} initializes interpolation sets in the same way as \texttt{COBYQA}~\cite{COBYQA}. Suppose $\mathcal{Y}_0 = \{y_0^1,\ldots,y_0^{N}\}$ is the $n$-dimensional interpolation set to be initialized with a capacity of $N = 2n+1$, the initial trust-region radius is $\Delta_0$, and the user-specified starting point is $x_{\text{start}}$. Then, the points $y_0^1,\ldots,y_0^{N}$ are initialized as  
\begin{equation*}
    y_0^i \triangleq
    \begin{cases}
    x_{\text{start}},& \text{if }~ i=1, \\
    x_{\text{start}}+\Delta_0 e_{i-1}, & \text{if }~2\leq i \leq n+1,\\
    x_{\text{start}}-\Delta_0 e_{i-n-1}, & \text{if }~ n+2\leq i \leq 2n+1.
    \end{cases}
\end{equation*}
Similarly, for each elemental interpolation set $\mathcal{Y}_{0,i}$, \texttt{UPOQA} initializes it with $x_{\text{start}}^{\mathcal{I}_i}$ as the center and $\Delta_0$ as the radius. Based on the resulting sets $\{\mathcal{Y}_{0,i}\}_{i=1}^q$, \texttt{UPOQA} then explores a certain number of solutions of a constraint satisfaction problem (CSP) using the minimum remaining values heuristic. In this CSP, the state is a complete point $x \in \Real^n$, and the constraints require that for any $i = 1,\ldots,q$, $P_{\mathcal{R}_i}(x) \in \mathcal{Y}_{0,i}$ must hold. Since the starting point search mechanism has limited impact on improving algorithmic performance, we restrict the search to at most 100 solutions and select the point with the smallest objective function value as the actual starting point $x_0$ for \texttt{UPOQA}.

\subsection{Optional Features}

Besides the procedure demonstrated in Algorithm \ref{alg:UPOQA}, \texttt{UPOQA} also provides two additional functionalities to enhance its flexibility and applicability: restart mechanism and hybrid black-white-box optimization.

\subsubsection{Restart Mechanism}

In model-based derivative-free methods, the trust region radius $\Delta$ or trust region resolution $\rho$ decreasing to a minimal value often indicates algorithm convergence. However, when the objective function contains noise, the noise will dominate the function value fluctuations once $\Delta$ becomes sufficiently small. Consequently, the interpolation model will no longer approximate the objective but model the noise instead. This leads to ineffective descent steps and may even trigger numerical instabilities causing algorithm failure.

Restart mechanisms have been widely adopted in other areas of numerical optimization, such as the conjugate gradient method and \texttt{GMRES}. Cartis et al.'s \texttt{Py-BOBYQA}~\cite{pybobyqa, pybobyqa_restart} introduced this mechanism into model-based DFO methods, with experiments demonstrating its robustness for noise. Inspired by this, \texttt{UPOQA} incorporates the soft restart mechanism similar to that in \texttt{Py-BOBYQA}. This mechanism is only activated upon user request. When numerical singularity or trust-region resolution depletion occurs, the algorithm resets $\Delta$ and $\rho$ to larger values, selectively replaces some existing interpolation points with a few randomly generated ones, and relocates the iterate to a distant position before resuming execution. 

\subsubsection{Hybrid Black-White-Box Optimization}

Real-world partially separable problems may exhibit more complex structures than (\ref{eq:PSProb}). To maximize the algorithm's applicability, we adapt the implementation of \texttt{UPOQA} to handle problems of the form  
\begin{equation}\label{eq:extended_CPS}
    \min_{x\in\Real^n}\ f(x) = f_0(x) + \sum_{i=1}^{q}w_i h_i\left(f_i(x)\right),
\end{equation}
where $f_1, \ldots, f_q$ remain black-box functions as before, while $f_0\colon\Real^n \to \Real$ is a white-box function whose gradient and Hessian can be calculated directly. Each $h_i\colon\Real \to \Real$ is a transformation applied to the element value $f_i(x)$, also treated as white-box, and $w_i \in \Real$ denotes the weight for each element. The transition from (\ref{eq:PSProb}) to (\ref{eq:extended_CPS}) requires minimal algorithmic modifications, as the added white-box components can be directly incorporated into the trust-region subproblem construction, yet it significantly enhances flexibility. Moreover, \texttt{UPOQA} allows users to dynamically update all $w_1, \ldots, w_q$ and $h_1, \ldots, h_q$ via a callback function during execution. This enables real-time trade-offs among multiple elements in multi-objective optimization, as well as effortless adjustment of penalty coefficients in constrained optimization without manual restarting.  

An additional advantage of this adaptation stems from \texttt{UPOQA}'s use of quadratic models for each black-box $f_i$. If $f_i$ inherently and locally resembles a quadratic function, maintaining an accurate surrogate model becomes computationally inexpensive, reducing unnecessary trial-and-error during optimization. Hence, knowledgeable users may preemptively extract known non-quadratic composite parts from $f_i$ and pass them as $h_i$ to the algorithm. This strategy lowers the cost of maintaining interpolation models and ultimately decreases the number of function evaluations required for convergence.

\section{Numerical Experiments}\label{sec:upoqa_experiment}

We briefly describe our testing methodology, the preparation of test problems, and the results of numerical experiments. Our primary evaluation metrics are the performance profile and data profile proposed by Moré and Wild~\cite{more2009benchmarking}, supplemented with an additional speed-up profile to assess the acceleration effect of \texttt{UPOQA} when exploiting partially separable structures. The \texttt{UPOQA} algorithm has been implemented as a Python library named \texttt{upoqa}. We conducted tests on benchmark problems extracted from \texttt{CUTEst}~\cite{gould2015cutest} using the \texttt{S2MPJ}~\cite{s2mpj} tool, along with a set of quantum variational problems. All tests were conducted on a server equipped with an AMD EPYC 9754 128-Core Processor and 1511.5 GiB of RAM. 

\subsection{Testing Methodology}

Let $\mathcal{S}$ denote the set of all solvers participating in the test, and $\mathcal{P}$ denote the set of test problems used. Each test problem $p\in \mathcal{P}$ has a fixed starting point $x_{0,p}$ and an objective function $f_p$. Given a tolerance $\varepsilon \in (0,1)$, a solver $s\in\mathcal{S}$ is said to converge on problem $p$ with tolerance $\varepsilon$ if and only if its iterate $x$ satisfies
\begin{equation}\label{eq:converge_test}
    f_p(x) \leq f_p^* + \varepsilon \left[f_p(x_{0,p}) - f_p^*\right],
\end{equation}
where $f_p^*$ is the lowest function value achieved by any solver in $\mathcal{S}$ on problem $p$. Let $t_{p,s}$ denote the number of function evaluations required for solver $s$ to converge on problem $p$. If convergence is not achieved, we set $t_{p,s}=\infty$.

The \emph{performance profile}~\cite{more2009benchmarking} is a concept introduced to evaluate and compare different optimization algorithms. This metric assesses which solver has an advantage in computational cost for given test problems. First, the performance ratio is defined as
\begin{equation*}
    r_{p,s} \triangleq
    \frac{t_{p,s}}{\min_{u \in \mathcal{S}}\{t_{p,u}\}},
\end{equation*}
which can be viewed as a relative cost incurred by solver $s$ to solve problem $p$. Then, the performance profile is defined as
\begin{equation*}
    \rho_s(\alpha) \triangleq \frac{1}{\operatorname{card}(\mathcal{P})} \operatorname{card}\left(\left\{p \in \mathcal{P}: r_{p, s} \leq \alpha\right\}\right), \quad \text{for}~ \alpha \geq1,
\end{equation*}
where $\operatorname{card}(\cdot)$ denotes the cardinality of a set. The metric $\rho_s(\alpha)$ can be interpreted as the proportion of problems in $\mathcal{P}$ that solver $s$ can solve with a performance ratio not exceeding $\alpha$. Specifically, $\rho_s(1)$ represents the proportion of problems where solver $s$ converges faster than all other solvers, while $\rho_s(+\infty)$ indicates the proportion of problems that $s$ can solve regardless of cost. Given two solvers $s_1$ and $s_2$, $\rho_{s_1}(\alpha) > \rho_{s_2}(\alpha)$ implies that under the constraints $r_{p,s_1} \leq \alpha$ and $r_{p,s_2} \leq \alpha$, $s_1$ can successfully solve more problems than $s_2$. The value of $\rho_s(\alpha)$ for an individual solver is highly dependent on the competitors present in $\mathcal{S}$. 

In contrast, the \emph{data profile}~\cite{more2009benchmarking} focuses more on the absolute capability of a single solver and is defined as  
\begin{equation*}  
    d_s(\alpha) \triangleq \frac{1}{\operatorname{card}(\mathcal{P})} \operatorname{card}\left(\left\{p \in \mathcal{P}: \frac{t_{p, s}}{n_p+1} \leq \alpha\right\}\right), \quad \text{for}~ \alpha \geq0,  
\end{equation*}  
where $n_p$ denotes the dimension of problem $p$. The metric $d_s(\alpha)$ can be interpreted as the proportion of problems in $\mathcal{P}$ that solver $s$ can solve using no more than $\alpha(n_p + 1)$ function evaluations. Specifically, $d_s(0) = 0$, while $d_s(+\infty)$ represents the proportion of problems that $s$ can solve regardless of computational cost. Given two solvers $s_1$ and $s_2$, $d_{s_1}(\alpha) > d_{s_2}(\alpha)$ implies that, for each problem $p$, when using no more than $\alpha(n_p + 1)$ function evaluations, $s_1$ can successfully solve more problems than $s_2$.  

When \texttt{UPOQA} converges on a given problem, the number of function evaluations for each element may differ. This discrepancy may arise from variations in the number of interpolation points required to initialize the surrogate models or from whether geometry improvement steps are performed on each element model during iterations. Let $t_{p,s}^{(i)}$ denote the number of function evaluations required for element $f_i$ when solver $s$ converges on problem $p$, $t_{p,s}^{\text{wst}}$ the maximum among $t_{p,s}^{(1)}, \ldots, t_{p,s}^{(q)}$, and $t_{p,s}^{\text{avg}}$ their average. If the cost of evaluating the objective function is dominated by a specific element $f_{i_*}$, the overall cost of the algorithm will be determined by $t_{p,s}^{(i_*)}$. Hence, in the most pessimistic scenario, the cost should be measured by $t_{p,s}^{\text{wst}}$ rather than $t_{p,s}^{\text{avg}}$. This is the meaning of the superscript $\text{wst}$ (worst-case). Unless otherwise specified, the numerical experiments in subsequent sections generally use $t_{p,s}^{\text{wst}}$ to quantify the computational cost of \texttt{UPOQA}. When no ambiguity arises, $t_{p,s}^{\text{wst}}$ may be abbreviated as $t_{p}^{\text{wst}}$.

Since \texttt{UPOQA} leveraging partially separable structures is often significantly faster than algorithms that do not utilize such structures—potentially rendering comparisons between the two less meaningful—this paper introduces an alternative metric called the \emph{speed-up profile}, defined as  
\begin{equation*}  
    \mathrm{su}(\alpha) \triangleq 
    \begin{cases}
        \frac{1}{\operatorname{card}(\mathcal{P})} \operatorname{card}\left(\left\{  
        p \in \mathcal{P}:   
        1 \leq c_p \leq \alpha \land \left(  
         t_{p}^{\text{single}} < \infty \lor  
        t_{p}^{\text{wst}} < \infty  
    \right)\right\}\right),&\quad \text{for}~ \alpha \geq 1, \\ 
    \frac{1}{\operatorname{card}(\mathcal{P})} \operatorname{card}\left(\left\{  
        p \in \mathcal{P}:   
        \alpha \leq c_p < 1 \land \left(   
         t_{p}^{\text{single}} < \infty \lor
        t_{p}^{\text{wst}} < \infty  
    \right)\right\}\right),&\quad \text{for}~ 0 \leq \alpha < 1,
    \end{cases}
\end{equation*}  
where 
\begin{equation}\label{eq:rela_speedup}
    c_p \triangleq \frac{  
            t_{p}^{\text{single}} / t_{p}^{\text{wst}}   
        }{  
            n_p ~/ \max\limits_{i=1,\ldots,q_p} n_{p,i}  
        }
\end{equation}
and \( t_{p}^{\text{single}} \) represents the number of function evaluations needed when the objective is treated as a single element and passed to \texttt{UPOQA}. Here, \( n_{p,i} \) is the dimension of the element $f_i$ for problem \( p \). Thus, \( t_{p}^{\text{single}} / t_{p}^{\text{wst}} \) measures the actual speed-up achieved by exploiting the partially separable structure, while \( n_p / \max\limits_{i=1,\ldots,q_p} n_{p,i} \) serves as a predicted speed-up ratio. To illustrate, consider an $n$-dimensional objective function formed by summing $n/\hat{n}$ low-dimensional subfunctions, each depending on distinct $\hat{n}$ variables. The Hessian is block-diagonal with identical $\hat{n} \times \hat{n}$ blocks. For gradient-based methods using forward-difference numerical gradients, exploiting this structure reduces per-iteration function evaluations from $O(n)$ to $O(\hat{n})$—each subgradient requires $\hat{n}+1$ evaluations rather than $n+1$ for the full gradient. Assuming comparable iteration counts, this induces a predicted speed-up ratio of $n/\hat{n}$. For general problems, the elemental dimensions \( n_1, \ldots, n_q \) may vary. In such cases, we assume that the element with largest dimension dominates the cost, making \( n~/ \max\limits_{i=1,\ldots,q_p} n_{i} \) a conservative prediction of the speed-up ratio. Notably, if any element has the same dimension as the objective function (i.e., \( \max\limits_{i=1,\ldots,q_p} n_{i} = n \)), \texttt{UPOQA} may fail to exploit any intrinsic low-dimensionality, resulting in a predicted speed-up ratio close to 1. 

The expression (\ref{eq:rela_speedup}) can be interpreted as a relative speed-up ratio, indicating whether \texttt{UPOQA} delivers the expected performance gain. A relative speed-up ratio greater than 1 suggests that revealing the partially separable structure not only allows the algorithm to exploit intrinsic low-dimensionality for reduced modeling cost but also enables the distillation of the primary nonconvexity into relatively low-dimensional elements, further lowering the actual cost. If \texttt{UPOQA} fails to converge regardless of whether the partially separable structure is utilized, such cases are excluded when computing \( \mathrm{su}(\alpha) \). 

To present the results more intuitively, we design \(\mathrm{su}(\alpha)\) to be computed across two distinct regimes: \(\alpha \geq 1\) and \(0 \leq \alpha < 1\). For \(\alpha \geq 1\), \(\mathrm{su}(\alpha)\) represents the proportion of problems where \texttt{UPOQA}, after exploiting the partially separable structure, achieves a relative speed-up ratio \(c_p \geq 1\) while maintaining \(c_p \leq \alpha\). A higher value in this regime is desirable, and the corresponding curve in the speed-up profile plot should ideally cluster toward the upper-left region. Specifically, \(\mathrm{su}(+\infty)\) quantifies the proportion of problems with \(c_p \geq 1\). For the regime \(0 \leq \alpha < 1\), \(\mathrm{su}(\alpha)\) measures the proportion of problems where \(c_p < 1\) yet \texttt{UPOQA} attains a relative speed-up ratio no less than \(\alpha\), while \(\mathrm{su}(0)\) indicates the proportion of problems with \(c_p < 1\). 

Furthermore, we always expect the plotted speed-up profile curve to exhibit a roughly balanced distribution across $\alpha = 1$, and to rise rapidly on both sides of it. This implies that for the majority of problems, the predicted speed-up ratio \( n_p / \max\limits_{i=1,\ldots,q_p} n_{p,i} \) approximates the actual speed-up ratio achieved.

\subsection{Test Results on \texttt{CUTEst} Problems}\label{sec:exp_results_on_cutest}

\subsubsection{Test Problems}

Since its release in 1995, the \texttt{CUTEst} testing environment and problem set~\cite{gould2015cutest} has been widely used by researchers in numerical optimization for designing, testing, and comparing both constrained and unconstrained optimization algorithms. Each test problem is described in the \emph{standard input format} (SIF), which utilizes the group partially separable structure~\cite{LANCELOT} defined as:
\begin{equation}\label{eq:groupPS}
   {f(x)}=\sum_{i \in \mathcal{G}_{\text {obj }}} \frac{F_i\left(a_i(x), \omega_i\right)}{\sigma_i}+\frac{1}{2} x^\top H x,
\end{equation}
where each term $F_i\left(a_i(x), \omega_i\right)/\sigma_i$ is called an objective-function group, $H\in \Real^{n\times n}$ is symmetric, $\mathcal{G}_{\text{obj}}$ is an index set, and $F_i$ is a linear or nonlinear function of $a_i(x)$ and $\omega_i$. We employ the \texttt{S2MPJ} tool~\cite{s2mpj} as the parser for SIF files to create partially separable versions of these problems with Python based on the group partially separable structure (\ref{eq:groupPS}). In our implementation, each $F_i\left(a_i(x), \omega_i\right)/\sigma_i$ forms an element, and the quadratic term $x^\top H x/2$ constitutes a separate element, together forming the partially separable structure of a problem. 

We conducted a screening of available \texttt{CUTEst} problems, excluding those that are inconvenient for testing or have undesirable structures, including:
\begin{itemize}
    \item Constrained problems;
    \item Problems with large average elemental dimension relative to the problem dimension, i.e., $\overline{n} = \left(\sum_{i=1}^q n_i\right)/q > n/2$, indicating a lack of intrinsic low-dimensionality;
    \item Problems with too many ($q > 3n$) or too few ($q < 3$) elements;
    \item Problems requiring excessively long construction time;
    \item Problems that are highly difficult to optimize, where none of the solvers in $\mathcal{S}$ can converge within the given evaluation budget;
    \item Problems with dimensions outside the medium-scale regime, i.e., $n < 20$ or $n > 100$.
\end{itemize}

Subsequently, we selected a test set of medium-scale problems, whose basic information is provided in Table \ref{tab:cutest_probs} in Appendix \ref{app:probs}. This set comprises 85 partially separable problems with dimensions $n \in [21,~100]$, most of which are 50-dimensional. The number of elements ranges from $15$ to $151$, with $q/n \in [0.17,~3.02]$, and $q \geq n$ holds for 70 problems. The partially separable structures of these \texttt{CUTEst} problems are derived from the group partially separable formulation (\ref{eq:groupPS}) of their SIF descriptions and are therefore semi-artificial: the decomposition reflects how each problem was modeled in SIF rather than an intrinsic structure of an underlying application, and the prevalent regime $q \geq n$ differs from many real-world problems. The quantum variational problems discussed later fall precisely into the regime $q \ll n$, arising from a practical application with a genuine partially separable structure. We hence view the two test sets as complementary.

\subsubsection{Test Results}

We tested the \texttt{UPOQA}, \texttt{UPOQA (single)}, \texttt{NEWUOA}, \texttt{COBYQA}, and \texttt{L-BFGS-B (ffd)} algorithms on the 85 problems included in Table \ref{tab:cutest_probs} in Appendix \ref{app:probs}. Here, \texttt{UPOQA (single)} denotes optimizing the objective function as a whole element without exploiting the partially separable structure. The \texttt{NEWUOA}~\cite{NEWUOAsoftware} implementation was from version 2.8.0 of Python's \texttt{nlopt} library, \texttt{COBYQA}~\cite{COBYQA} was from version 1.1.1 of the \texttt{COBYQA} package, and \texttt{L-BFGS-B} was from version 1.14.1 of the \texttt{scipy} library, using forward-finite-difference gradients for the objective with a default differential step size of $1.49\times 10^{-8}$. For each problem, the maximum number of function evaluations allowed was set to $\max\{1000 n, 10000\}$, with the same limit applied to each element function evaluation in \texttt{UPOQA}. For the convergence criterion (\ref{eq:converge_test}), we selected $\varepsilon = 10^{-1},~10^{-3},~10^{-5}$, and $10^{-7}$, obtaining the performance profiles in Figure \ref{fig:perf_profiles_main_problems} and the data profiles shown in Figure \ref{fig:data_profiles_main_problems}. Table \ref{tab:5_prob_result} lists the number of function evaluations required for convergence on the five problems of largest dimensions in the test set, where $\infty$ indicates failure to converge within the evaluation budget.


\renewcommand{\arraystretch}{1.2}

\begin{table}
    \caption{Function evaluations required for convergence by \texttt{UPOQA}, \texttt{UPOQA (single)}, \texttt{NEWUOA}, \texttt{COBYQA}, and \texttt{L-BFGS-B (ffd)} on the five largest-dimensional problems from the test set in Table \ref{tab:cutest_probs} in Appendix \ref{app:probs}, with convergence tolerances $\varepsilon = 10^{-1},~10^{-3},~10^{-5}$, and $10^{-7}$. For \texttt{UPOQA}, the reported evaluations are the maximum evaluations across all elements. Here, $n$ denotes the problem dimension, $q$ denotes the number of elements in its partially separable structure, $n_1,\ldots,n_q$ are dimensionalities of elements, and $\infty$ indicates failure to converge within the evaluation budget. The last row of each problem reports the average wall-clock time per function evaluation, where the total number of evaluations is the number used by the algorithm when it terminated.}
    \centering
    \footnotesize
    \setlength{\tabcolsep}{3pt}
    {\begin{tabular}{ccccccccccc}
    \Xhline{1pt}

    \multirow{2}{*}{Problem}        & \multirow{2}{*}{$n$} & \multirow{2}{*}{$q$} &  \multirow{2}{*}{$\max n_i$} & \multirow{2}{*}{$\sum n_i /q$} & \multirow{2}{*}{$\log_{10}\varepsilon~$} &  \multicolumn{5}{c}{Evaluations}          \\ \Xcline{7-11}{0.3pt} 
    &      &    &  &   &   &  \texttt{UPOQA} & \texttt{UPOQA (single)} & \texttt{NEWUOA} & \texttt{COBYQA} & \texttt{L-BFGS-B (ffd)} \\

    \Xhline{0.7pt}
    \multirow{5}{*}[-0.7em]{HIMMELBI} & \multirow{5}{*}[-0.7em]{100} & \multirow{5}{*}[-0.7em]{20} & \multirow{5}{*}[-0.7em]{5} & \multirow{5}{*}[-0.7em]{5.00} & -1 & \textbf{24}   & 380           & 370   & 360   & 506          \\
                            &                     &               & & & -3                   & \textbf{102}   & 2720           & 1400   & 1772   & 1213          \\
                            &                     &               & & & -5                   & \textbf{171}   & 3496           & 2437   & 2578   & 1920          \\
                            &                     &               & & & -7                   & \textbf{273}   & 5648           & 3305   & 3552   & $\infty$         \\ \Xcline{6-11}{0.2pt}
                            &                     &               & & & Time/Eval           & \begin{tabular}{@{}c@{}}0.149s\\(87.4s/588)\end{tabular}   & \begin{tabular}{@{}c@{}}0.006s\\(47.3s/7568)\end{tabular}   & \begin{tabular}{@{}c@{}}0.002s\\(55.8s/22987)\end{tabular}   & \begin{tabular}{@{}c@{}}0.629s\\(2.9ks/4561)\end{tabular}   & \begin{tabular}{@{}c@{}}0.001s\\(2.7s/2424)\end{tabular}         \\
    \hline
    \multirow{5}{*}[-0.7em]{JANNSON3} & \multirow{5}{*}[-0.7em]{100} & \multirow{5}{*}[-0.7em]{102} & \multirow{5}{*}[-0.7em]{2} & \multirow{5}{*}[-0.7em]{1.01} & -1 & \textbf{5}   & 204           & 204   & 205   & 203          \\
                            &                     &               & & & -3                   & \textbf{7}   & 219           & 253   & 237   & 506          \\
                            &                     &               & & & -5                   & \textbf{10}   & 420           & 781   & 418   & 708          \\
                            &                     &               & & & -7                   & \textbf{16}   & 643           & 1094   & 569   & 809          \\ \Xcline{6-11}{0.2pt}
                            &                     &               & & & Time/Eval           & \begin{tabular}{@{}c@{}}0.053s\\(1.9s/36)\end{tabular}   & \begin{tabular}{@{}c@{}}0.010s\\(11.0s/1079)\end{tabular}   & \begin{tabular}{@{}c@{}}0.007s\\(117.8s/18005)\end{tabular}   & \begin{tabular}{@{}c@{}}0.416s\\(377.7s/908)\end{tabular}   & \begin{tabular}{@{}c@{}}0.005s\\(5.3s/1010)\end{tabular}          \\
    \hline
    \multirow{5}{*}[-0.7em]{LUKSAN21LS} & \multirow{5}{*}[-0.7em]{100} & \multirow{5}{*}[-0.7em]{100} & \multirow{5}{*}[-0.7em]{3} & \multirow{5}{*}[-0.7em]{2.98} & -1 & \textbf{523}   & 34455           & 20453   & 31016   & 40704          \\
                            &                     &               & & & -3                   & \textbf{541}   & 40583           & 25593   & 35370   & 45047          \\
                            &                     &               & & & -5                   & \textbf{560}   & 48012           & 32122   & 42446   & 50097          \\
                            &                     &               & & & -7                   & \textbf{580}   & 57425           & 38317   & 46248   & 55854          \\ \Xcline{6-11}{0.2pt}
                            &                     &               & & & Time/Eval           & \begin{tabular}{@{}c@{}}0.777s\\(533.6s/687)\end{tabular}   & \begin{tabular}{@{}c@{}}0.011s\\(798.9s/69766)\end{tabular}   & \begin{tabular}{@{}c@{}}0.008s\\(729.7s/96699)\end{tabular}   & \begin{tabular}{@{}c@{}}0.546s\\(30.6ks/56145)\end{tabular}   & \begin{tabular}{@{}c@{}}0.005s\\(329.3s/60095)\end{tabular}          \\
    \hline
    \multirow{5}{*}[-0.7em]{LINVERSE} & \multirow{5}{*}[-0.7em]{99} & \multirow{5}{*}[-0.7em]{147} & \multirow{5}{*}[-0.7em]{4} & \multirow{5}{*}[-0.7em]{2.98} & -1 & \textbf{10}   & 204           & 218   & 207   & 301          \\
                            &                     &               & & & -3                   & \textbf{46}   & 1520           & 1380   & 1304   & 1101          \\
                            &                     &               & & & -5                   & \textbf{95}   & 3876           & 2975   & 2681   & 2501          \\
                            &                     &               & & & -7                   & \textbf{112}   & 8230           & 5220   & 3774   & 3801          \\ \Xcline{6-11}{0.2pt}
                            &                     &               & & & Time/Eval           & \begin{tabular}{@{}c@{}}1.568s\\(526.8s/336)\end{tabular}   & \begin{tabular}{@{}c@{}}0.009s\\(149.7s/15959)\end{tabular}   & \begin{tabular}{@{}c@{}}0.005s\\(120.7s/25009)\end{tabular}   & \begin{tabular}{@{}c@{}}0.499s\\(14.1ks/28171)\end{tabular}   & \begin{tabular}{@{}c@{}}0.003s\\(19.1s/5600)\end{tabular}          \\
    \hline
    \multirow{5}{*}[-0.7em]{HYDC20LS} & \multirow{5}{*}[-0.7em]{99} & \multirow{5}{*}[-0.7em]{99} & \multirow{5}{*}[-0.7em]{14} & \multirow{5}{*}[-0.7em]{7.47} & -1 & \textbf{133}   & 528           & 471   & 655   & 1001          \\
                            &                     &               & & & -3                   & \textbf{1322}   & 14818           & 8456   & 10809   & 13801          \\
                            &                     &               & & & -5                   & \textbf{13958}   &  $\infty$            &  $\infty$    &  $\infty$    &  $\infty$           \\
                            &                     &               & & & -7                   & \textbf{22350}   &  $\infty$           &  $\infty$    &  $\infty$    &  $\infty$           \\ \Xcline{6-11}{0.2pt}
                            &                     &               & & & Time/Eval           & \begin{tabular}{@{}c@{}}0.888s\\(41.5ks/46720)\end{tabular}   & \begin{tabular}{@{}c@{}}0.014s\\(1.4ks/99000)\end{tabular}   & \begin{tabular}{@{}c@{}}0.009s\\(887.0s/99000)\end{tabular}   & \begin{tabular}{@{}c@{}}0.572s\\(56.7ks/99000)\end{tabular}   & \begin{tabular}{@{}c@{}}0.007s\\(688.1s/99100)\end{tabular}           \\
    \Xhline{1pt}
\end{tabular}}
\label{tab:5_prob_result}
\end{table}

\begin{figure}
    \centering
    \subfloat[$\varepsilon = 10^{-1}$]{\includegraphics[width = 0.45\textwidth]{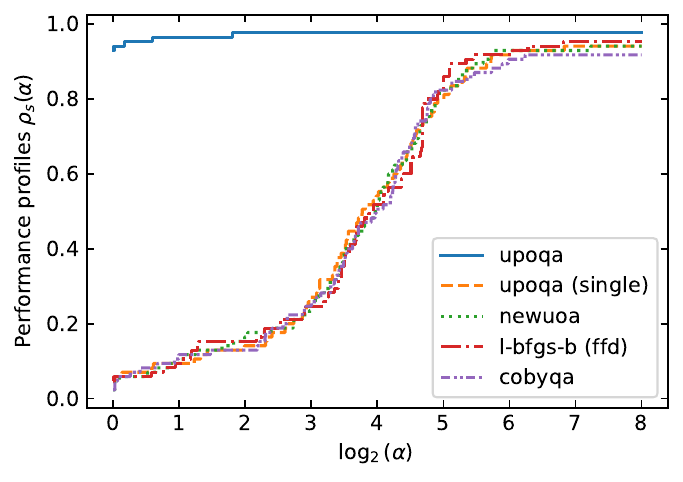}}
    \hfill
    \subfloat[$\varepsilon = 10^{-3}$]{\includegraphics[width = 0.45\textwidth]{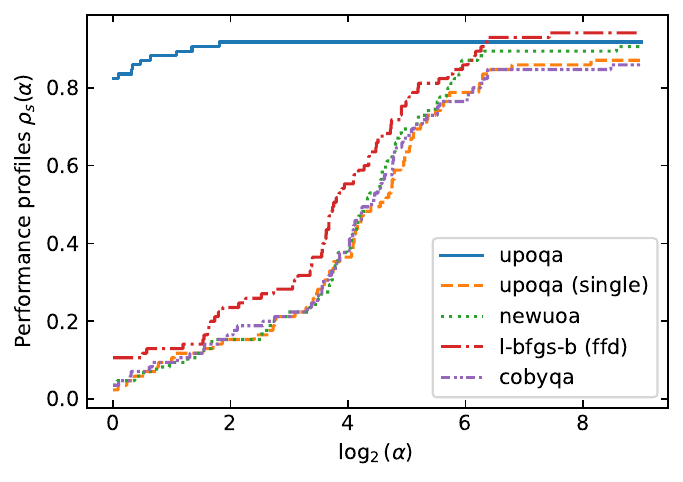}}

    \subfloat[$\varepsilon = 10^{-5}$]{\includegraphics[width = 0.45\textwidth]{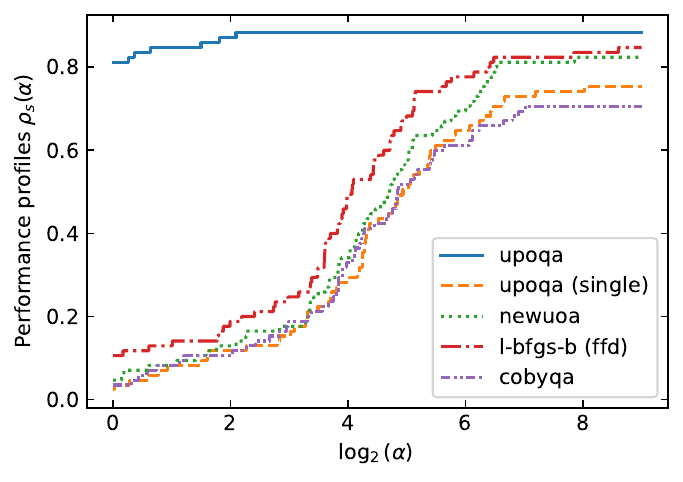}}
    \hfill
    \subfloat[$\varepsilon = 10^{-7}$]{\includegraphics[width = 0.45\textwidth]{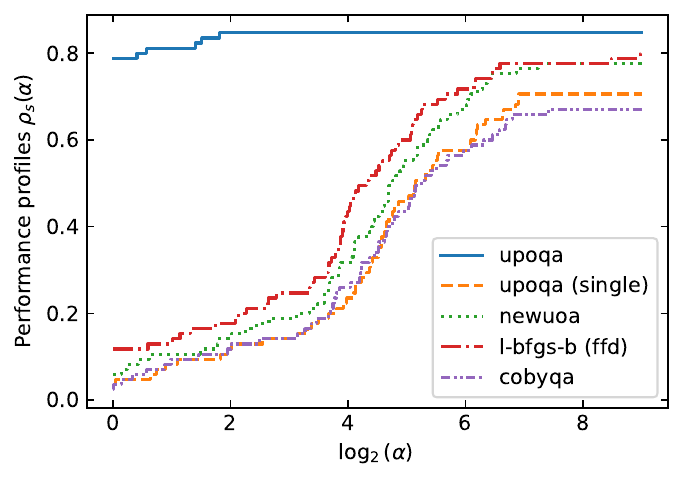}}
    \caption{
        Performance profiles of \texttt{UPOQA}, \texttt{UPOQA (single)}, \texttt{NEWUOA}, \texttt{COBYQA}, and \texttt{L-BFGS-B (ffd)} on the test problems from Table \ref{tab:cutest_probs} in Appendix \ref{app:probs}, with convergence tolerances $\varepsilon = 10^{-1},~10^{-3},~10^{-5}$, and $10^{-7}$.
    }
    \Description{Fully described in the text.}
    \label{fig:perf_profiles_main_problems}
\end{figure}

Figures \ref{fig:perf_profiles_main_problems} and \ref{fig:data_profiles_main_problems}, along with Table \ref{tab:5_prob_result}, demonstrate that \texttt{UPOQA} achieves significant computational acceleration by effectively exploiting the partially separable structure. At tolerance levels of $\varepsilon = 10^{-1}, 10^{-5}$ and $10^{-7}$, \texttt{UPOQA} achieves success rates of 97.6\%, 88.2\% and 84.7\% respectively, surpassing all competing solvers. When $\varepsilon = 10^{-3}$, \texttt{UPOQA} solves 91.8\% of problems, marginally lower than the 94.1\% of \texttt{L-BFGS-B (ffd)}, yet outperforming \texttt{UPOQA (single)}, \texttt{NEWUOA}, and \texttt{COBYQA}. Furthermore, for successfully solved problems, \texttt{UPOQA} requires substantially fewer function evaluations compared to other methods. While \texttt{UPOQA (single)} exhibits lower success rates than \texttt{NEWUOA} and \texttt{L-BFGS-B (ffd)}, this is expected as it is not specifically designed for single-objective optimization. Remarkably, for all tolerances, \texttt{UPOQA} solves at least 78.8\% of problems with fewest evaluations, reaching 92.9\% at $\varepsilon = 10^{-1}$. 

The last row of Table \ref{tab:5_prob_result} reports the average wall-clock time per function evaluation. Since the test problems have inexpensive evaluations, the measured time primarily reflects the internal computational overhead of each algorithm. \texttt{UPOQA} exhibits higher per-evaluation overhead due to its element-wise modeling, but its total wall-clock time remains competitive on most problems. In practical scenarios where function evaluations are expensive, the reduction in evaluation counts provided by \texttt{UPOQA} would dominate the total cost.

\begin{figure}
    \centering
    \subfloat[$\varepsilon = 10^{-1}$]{\includegraphics[width = 0.45\textwidth]{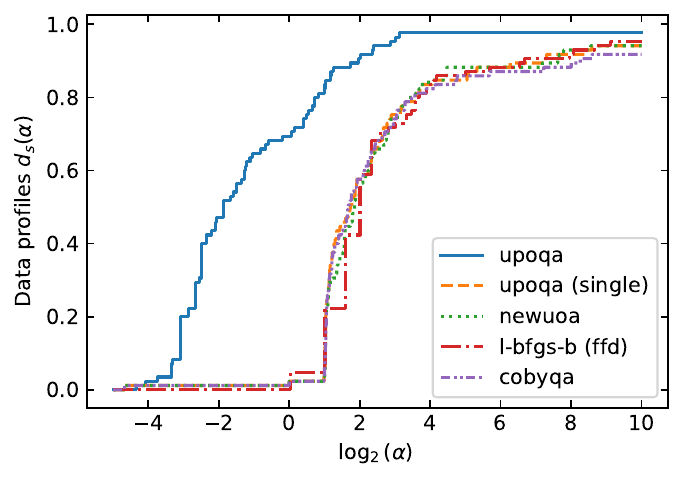}}
    \hfill
    \subfloat[$\varepsilon = 10^{-3}$]{\includegraphics[width = 0.45\textwidth]{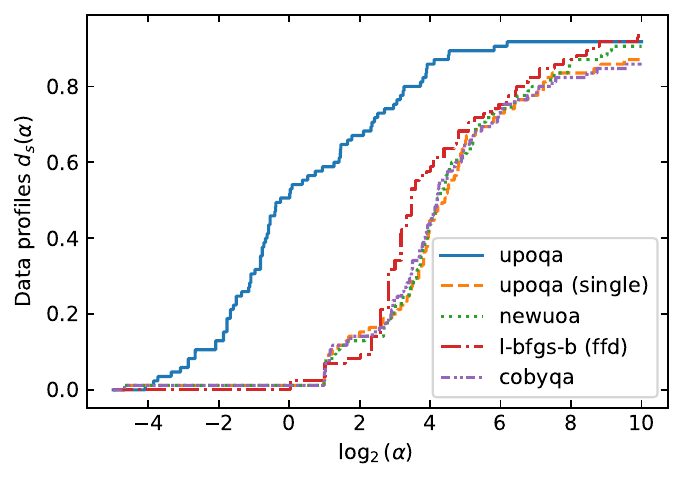}}

    \subfloat[$\varepsilon = 10^{-5}$]{\includegraphics[width = 0.45\textwidth]{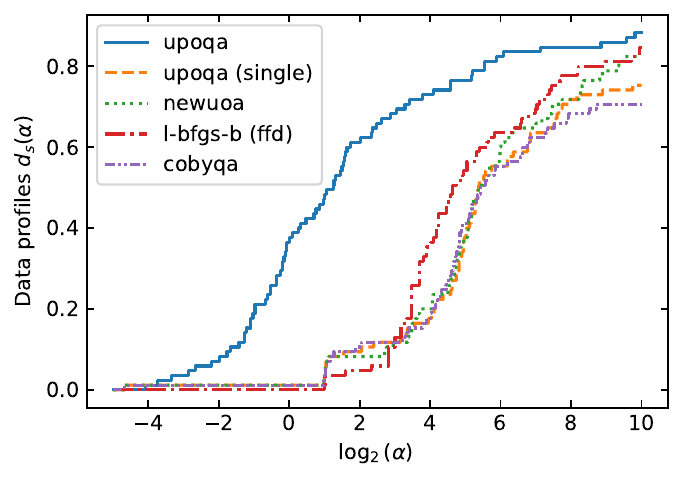}}
    \hfill
    \subfloat[$\varepsilon = 10^{-7}$]{\includegraphics[width = 0.45\textwidth]{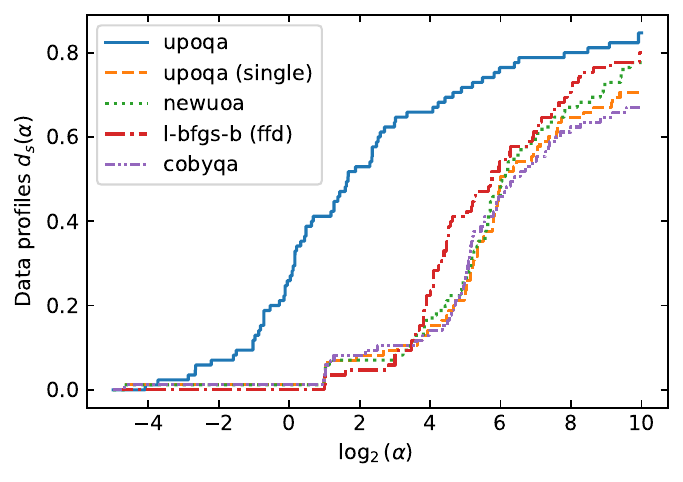}}
    \caption{Data profiles of \texttt{UPOQA}, \texttt{UPOQA (single)}, \texttt{NEWUOA}, \texttt{COBYQA}, and \texttt{L-BFGS-B (ffd)} on the test problems from Table \ref{tab:cutest_probs} in Appendix \ref{app:probs}, with convergence tolerances $\varepsilon = 10^{-1},~10^{-3},~10^{-5}$, and $10^{-7}$.}
    \Description{Fully described in the text.}

\label{fig:data_profiles_main_problems}
\end{figure}

To further evaluate the acceleration effect of \texttt{UPOQA}, we present the speed-up profiles computed by \texttt{UPOQA} and \texttt{UPOQA (single)} in Figure \ref{fig:speed-up_profiles_main_problems}. At a low convergence tolerance ($\varepsilon = 10^{-1}$), the profile shows that among problems where at least one algorithm converged, 70.0\% exhibited relative speed-up ratios between 0.5 and 2 (i.e., the gray-shaded region~$[-1,1]$ on the horizontal axis), indicating reasonable alignment between predictions and empirical results. At the same time, \texttt{UPOQA} achieved lower speed-up than predicted on 58.8\% of problems, outperformed predictions on 35.3\%, while failing to converge within the budget on the remaining 5.9\%—regardless of exploiting partially separable structures. Thus, \texttt{UPOQA}'s acceleration advantage is limited when only low-precision solutions are required. However, at higher precision ($\varepsilon = 10^{-4}$), its acceleration becomes significantly enhanced, with a notably increased (from 35.3\% to 50.6\%) proportion of problems distributed on the $\log_2(\alpha) \geq 0$ side,  and the proportion of problems with $\log_2(\alpha) < 0$ decreases markedly from 58.8\% to 28.2\%—a consequence of reduced solve success rates under stricter precision demands. When tolerance tightens further to $\varepsilon = 10^{-7}$, the outperformance ratio declines slightly to 45.9\%. Overall, the speed-up of \texttt{UPOQA} is less significant in low-precision scenarios but becomes more pronounced in high-precision contexts.

\begin{figure}
    \centering
    \includegraphics[width = 0.625\textwidth]{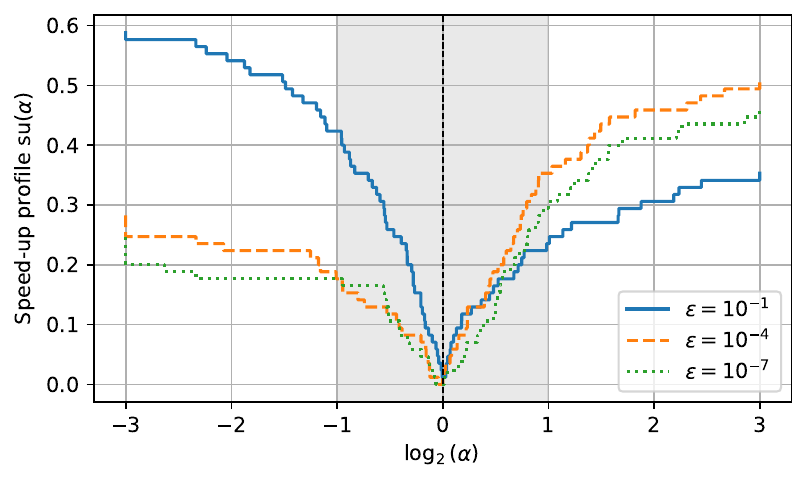}
    \caption{Speed-up profiles computed by \texttt{UPOQA} and \texttt{UPOQA (single)} on the test problems from Table \ref{tab:cutest_probs} in Appendix \ref{app:probs}, with convergence tolerances $\varepsilon = 10^{-1},~10^{-4}$, and $10^{-7}$.}
    \label{fig:speed-up_profiles_main_problems}
    \Description{Fully described in the text.}
\end{figure}

\subsection{Test Results on Quantum Variational Problems}

With the advancement of quantum technology, the \emph{variational quantum algorithm} (VQA) has garnered increasing attention as a hybrid quantum-classical approach for solving optimization problems~\cite{VQA}. In brief, given an objective function, the first step of VQA involves constructing a parameterized quantum circuit on a quantum computer to represent this function. Classical derivative-free optimization algorithms are then employed to adjust these parameters and ultimately find the minimum of the objective. Now, let us consider a quantum variational problem:
\begin{equation}\label{eq:quant_prob_intro}
\min_{\theta_1,\ldots,\theta_p\ \in\ \Real^{n_\theta} } \quad \sum_{i = 1}^p w_i\left\langle \phi_i \left| U(\theta_i)^\dagger \widehat{H} U(\theta_i) \right| \phi_i \right\rangle + \lambda
\sum_{i < j} \left|\langle \phi_i \left| U(\theta_i)^\dagger U(\theta_j) \right| \phi_j \rangle \right|^2,
\end{equation}
where $\widehat{H}$ denotes the Hamiltonian operator of the molecular system, $U(\theta_1),\ldots,U(\theta_p)$ represent the parameterized ansatz quantum circuits with parameters $\theta_1,\ldots,\theta_p$ respectively, $\phi_1,\ldots,\phi_p$ are the initial quantum states, $w_1,\ldots,w_p >0$ are the weights, and $\lambda > 0$ is the penalty coefficient. The goal of this problem is to determine the states corresponding to the smallest $p$ energy levels of a given molecular system.

In equation (\ref{eq:quant_prob_intro}), each term under the summation represents the measurement outcome of a quantum circuit, which is typically quite costly. We treat each term in (\ref{eq:quant_prob_intro}) as a black box and reformulate the problem as:
\begin{eqnarray}\label{eq:quant_prob_bb}
    \min_{\theta_1,\ \ldots,\ \theta_p\ \in\ \Real^{n_\theta} }
    \ \sum_{i=1}^p f_i(\theta_i) + 
    \sum_{1\leq i< j \leq p} f_{ij}(\theta_i,\theta_j),
\end{eqnarray}
where $f_i(\theta_i) = w_i\left\langle \phi_i \left| U(\theta_i)^\dagger \widehat{H} U(\theta_i) \right| \phi_i \right\rangle$, $f_{ij}(\theta_i,\theta_j) = \lambda \left|\left\langle \phi_i \left| U(\theta_i)^\dagger U(\theta_j) \right| \phi_j \right\rangle \right|^2$. Clearly, this is a coordinate partially separable problem, where the objective function is $np$-dimensional and consists of \( p(p+1)/2 \) elements with a maximum dimensionality of \( 2n \) and exhibits a dense Hessian matrix. Similar structures are common in other quantum variational problems~\cite{QOMM, QOMM2}.

We conduct numerical experiments on problem (\ref{eq:quant_prob_bb}) for \(\text{H}_2\) and a model consisting of four hydrogen atoms arranged in a square lattice, which is denoted as \(\text{H}_4\), in the STO-3G basis with UCCSD~\cite{UCCSD} ansatz circuits. The UCCSD block pattern is repeated for three times in order to increase the expressiveness of the ansatz. We set \( p = 3 \), indicating that the goal is to compute the three lowest-energy excited states and their state energies, with weights set as $w_i = i$ for $i=1, 2, 3$. The set of initial states \( \phi_1, \phi_2, \phi_3 \) and all molecular Hamiltonians are chosen to be the Hartree-Fock state and low-lying single-particle excitations above it. The starting points \( \theta_{0,1}, \theta_{0,2}, \theta_{0,3} \) are randomly sampled according to a uniform distribution on $[-2\pi, 2\pi)$. Both experiments are performed in a noiseless environment. The experimental code is implemented in Python using \texttt{qiskit} version 0.37.0, with quantum circuit designs consistent with those in the quantum orbital minimization method (\texttt{qOMM})~\cite{QOMM}. In the \(\text{H}_2\) experiment, the interatomic distance is set to 0.735 \AA, and the penalty coefficient is empirically chosen as \( \lambda = 12 \). For the \(\text{H}_4\) experiment, the interatomic distance is 1.23 \AA, with \( \lambda = 30 \).

\begin{figure}
    \centering
        \subfloat[$\text{H}_2$ molecule]{\includegraphics[width = 0.44\textwidth]{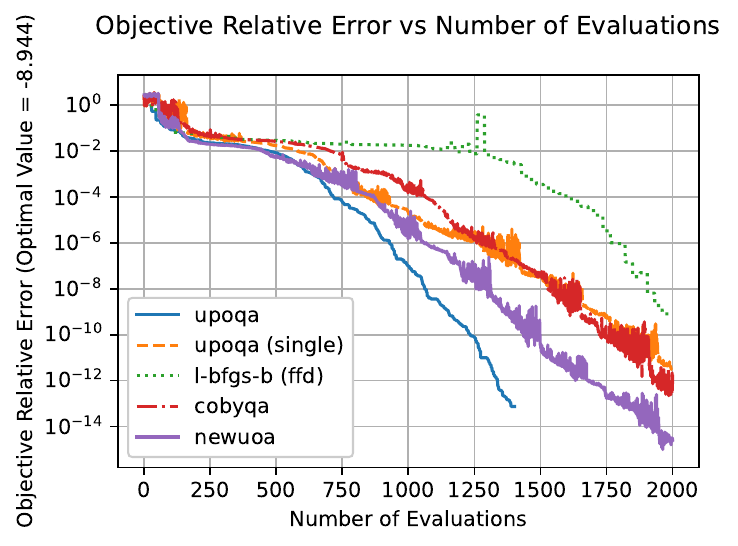}}
        \hfill
        \subfloat[$\text{H}_4$ molecule]{\includegraphics[width = 0.44\textwidth]{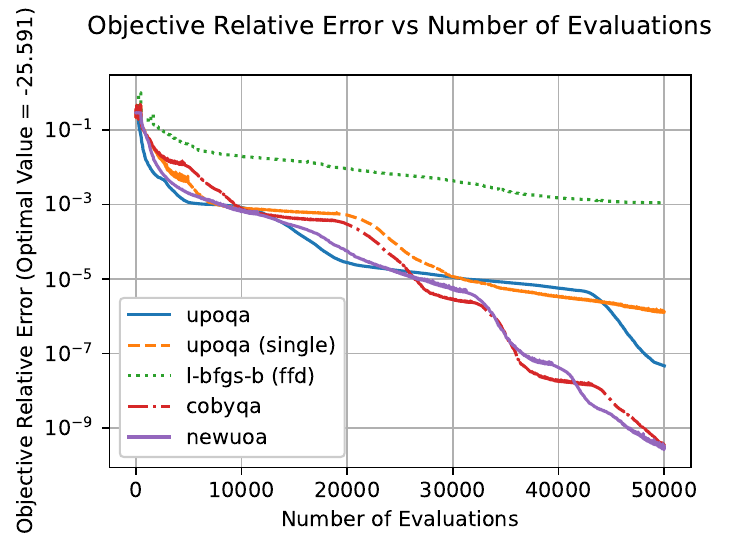}}
    \caption{Objective function value descent curves when optimizing the noiseless quantum variational problem (\ref{eq:quant_prob_bb}) for $\text{H}_2$ and $\text{H}_4$ molecular systems using \texttt{UPOQA}, \texttt{UPOQA (single)}, \texttt{NEWUOA}, \texttt{COBYQA}, and \texttt{L-BFGS-B (ffd)}.}
    \Description{Fully described in the text.}
    \label{fig:Q_problems}
\end{figure}

The experimental results from both tests are shown in Fig \ref{fig:Q_problems}, with maximum function evaluation budgets set to $2000$ and $50000$ for the $\text{H}_2$ and $\text{H}_4$ problems, respectively. On the $\text{H}_2$ problem, \texttt{UPOQA} achieves the fastest convergence, confirming that exploiting the partially separable structure indeed accelerates the optimization. On the larger $\text{H}_4$ problem, \texttt{NEWUOA} and \texttt{COBYQA} reach high precision fastest, while \texttt{UPOQA} remains competitive and clearly outperforms \texttt{UPOQA (single)} and \texttt{L-BFGS-B (ffd)}. This is possibly because the predicted acceleration rate $p/2 = 1.5$ and the substantial variable overlap among element functions limit the structural advantage of partial separability, though \texttt{UPOQA} still consistently outperforms \texttt{UPOQA (single)} on both problems. These results suggest that \texttt{UPOQA} could deliver stronger performance on larger-scale quantum variational problems with more favorable structural properties.

\section{Conclusion}\label{sec:conclusion}

We propose the \texttt{UPOQA} algorithm for derivative-free optimization of partially separable problems. Based on quadratic interpolation models and a trust-region framework, \texttt{UPOQA} constructs underdetermined quadratic models for each element function, employing the Steinmetz projection and a modified projected gradient method to solve structured trust-region subproblems. The implementation incorporates Powell's techniques~\cite{PowellUpdateInverseKKT, PowellDevNEWUOA}, resulting in low iteration costs and strong robustness. The trust-region radius management strategy is modified from Shahabuddin's criterion~\cite[][Section 2.4]{PSDFOThesis}. Additionally, \texttt{UPOQA} features practical functionalities such as starting point search, restart mechanism, and hybrid black-white-box optimization. Comprehensive numerical experiments on \texttt{CUTEst} problems and a set of quantum variational problems validate the algorithm's effectiveness and robustness in problem-solving.

We acknowledge several aspects of \texttt{UPOQA} that warrant further research and improvement. The modified projected gradient method currently exhibits limited single-step descent rates and requires excessive iterations, leading to significant runtime overhead in some scenarios. Theoretical analysis remains incomplete, particularly regarding convergence rates and their relationship with partially separable structures. Experimentally, we aspire to test the algorithm on more real-world applications, as the artificially constructed partially separable structures from \texttt{CUTEst} test problems may lack sufficient representativeness. Future enhancements could potentially incorporate state-of-the-art DFO techniques such as sample averaging and regression models~\cite{pybobyqa}, as well as preconditioning and subspace methods~\cite{NEWUOAs}, which might further improve \texttt{UPOQA}'s performance.

\bibliographystyle{ACM-Reference-Format}
\bibliography{paper}

\appendix

\section{Test Problems}\label{app:probs}

Table \ref{tab:cutest_probs} lists all the test problems used in Section \ref{sec:exp_results_on_cutest}. These problems were extracted from the \texttt{CUTEst} problem set \cite{gould2015cutest} using the \texttt{S2MPJ} tool \cite{s2mpj}.

\begin{table}
    \caption{\texttt{CUTEst} test problems used in Section \ref{sec:exp_results_on_cutest}, where $n$ denotes the problem dimension, $q$ denotes the number of elements, and $n_1,\ldots,n_q$ are the dimensions of elements.}
    \label{tab:cutest_probs}
    \footnotesize
    \begin{tabular}{rrrrr|rrrrr}
        \toprule
        Problem & $n$ & $q$ & $\max n_i$ & $\sum n_i /q$ & Problem & $n$ & $q$ & $\max n_i$ & $\sum n_i /q~$ \\ 
        \midrule
        ANTWERP & 27 & 19 & 19 & 7.74 & ARWHEAD & 50 & 98 & 2 & 1.50 \\
        BDQRTIC & 50 & 92 & 5 & 3.00 & BIGGSB1 & 50 & 51 & 2 & 1.96 \\
        BROYDN3DLS & 50 & 50 & 3 & 2.96 & BROYDNBDLS & 50 & 50 & 7 & 6.68 \\
        CHNROSNB & 50 & 98 & 2 & 1.50 & CHNRSNBM & 50 & 98 & 2 & 1.50 \\
        COSINE & 50 & 49 & 2 & 2.00 & CURLY10 & 50 & 50 & 11 & 9.90 \\
        CURLY20 & 50 & 50 & 21 & 16.80 & CURLY30 & 50 & 50 & 31 & 21.70 \\
        CVXQP1 & 50 & 50 & 3 & 2.94 & CYCLOOCFLS & 86 & 60 & 6 & 5.73 \\
        CYCLOOCTLS & 90 & 60 & 6 & 6.00 & DIXON3DQ & 50 & 50 & 2 & 1.96 \\
        DQRTIC & 50 & 50 & 1 & 1.00 & DTOC2 & 88 & 15 & 6 & 5.87 \\
        EG2 & 50 & 50 & 2 & 1.96 & ENGVAL1 & 50 & 98 & 2 & 1.50 \\
        ERRINROS & 50 & 98 & 2 & 1.50 & ERRINRSM & 50 & 98 & 2 & 1.50 \\
        EXTROSNB & 50 & 50 & 2 & 1.98 & FLETBV3M & 50 & 102 & 50 & 1.96 \\
        FLETCBV2 & 50 & 151 & 2 & 1.32 & FLETCHCR & 50 & 98 & 2 & 1.50 \\
        FREUROTH & 50 & 98 & 2 & 2.00 & GILBERT & 50 & 50 & 1 & 1.00 \\
        HATFLDC & 25 & 25 & 2 & 1.92 & HIMMELBI & 100 & 20 & 5 & 5.00 \\
        HUESmMOD & 50 & 50 & 1 & 1.00 & HYDC20LS & 99 & 99 & 14 & 7.47 \\
        HYDCAR6LS & 29 & 29 & 14 & 6.69 & INDEFM & 50 & 98 & 3 & 1.98 \\
        JANNSON3 & 100 & 102 & 2 & 1.01 & JANNSON4 & 50 & 52 & 2 & 1.02 \\
        LEVYMONT10 & 50 & 100 & 2 & 1.49 & LEVYMONT6 & 50 & 100 & 2 & 1.49 \\
        LIARWHD & 50 & 100 & 2 & 1.49 & LINVERSE & 99 & 147 & 4 & 2.98 \\
        LUKSAN21LS & 100 & 100 & 3 & 2.98 & LUKVLI10 & 50 & 50 & 2 & 2.00 \\
        LUKVLI11 & 50 & 64 & 2 & 1.25 & LUKVLI13 & 50 & 48 & 2 & 1.67 \\
        LUKVLI14 & 50 & 64 & 2 & 1.25 & LUKVLI1 & 50 & 98 & 2 & 1.50 \\
        LUKVLI2 & 50 & 125 & 2 & 1.62 & LUKVLI3 & 50 & 96 & 2 & 2.00 \\
        LUKVLI4C & 50 & 120 & 2 & 1.60 & LUKVLI5 & 52 & 50 & 3 & 3.00 \\
        LUKVLI8 & 50 & 40 & 5 & 4.00 & LUKVLI9 & 50 & 51 & 50 & 2.45 \\
        METHANB8LS & 31 & 31 & 11 & 6.29 & METHANL8LS & 31 & 31 & 11 & 6.29 \\
        MOREBV & 50 & 50 & 3 & 2.96 & NCB20B & 50 & 50 & 20 & 12.78 \\
        NCB20 & 60 & 51 & 30 & 12.75 & NONDQUAR & 50 & 50 & 3 & 2.96 \\
        NONSCOMP & 50 & 50 & 2 & 1.98 & OSCIGRAD & 50 & 50 & 3 & 2.96 \\
        OSCIPATH & 50 & 50 & 2 & 1.98 & PENALTY1 & 50 & 51 & 50 & 1.96 \\
        PENALTY2 & 50 & 100 & 50 & 1.98 & QING & 50 & 50 & 1 & 1.00 \\
        RAYBENDL & 62 & 30 & 4 & 4.00 & SANTALS & 21 & 23 & 4 & 3.52 \\
        SBRYBND & 50 & 50 & 7 & 6.68 & SCHMVETT & 50 & 48 & 3 & 3.00 \\
        SCOSINE & 50 & 49 & 2 & 2.00 & SINEALI & 50 & 50 & 2 & 1.98 \\
        SINQUAD2 & 50 & 50 & 3 & 2.94 & SINROSNB & 50 & 50 & 2 & 1.98 \\
        SPARSINE & 50 & 50 & 6 & 5.58 & SPARSQUR & 50 & 50 & 6 & 5.58 \\
        SSBRYBND & 50 & 50 & 7 & 6.68 & STRTCHDV & 50 & 49 & 2 & 2.00 \\
        TOINTGOR & 50 & 83 & 5 & 1.81 & TOINTGSS & 50 & 48 & 3 & 3.00 \\
        TOINTPSP & 50 & 83 & 5 & 1.81 & TOINTQOR & 50 & 83 & 5 & 1.81 \\
        TQUARTIC & 50 & 50 & 2 & 1.98 & TRIDIA & 50 & 50 & 2 & 1.98 \\
        VARDIM & 50 & 52 & 50 & 2.88 & VAREIGVL & 51 & 51 & 50 & 13.88 \\
        YAO & 52 & 52 & 1 & 1.00 & & & & & \\
        \bottomrule
    \end{tabular}
\end{table}

\section{Performance of \texttt{UPOQA} Using Unstructured Trust Regions}\label{sec:non_struct}

We also compared the performance of \texttt{UPOQA} with and without structured trust regions. When structured trust regions are not employed, all the radii $\Delta_{k,1},\ldots,\Delta_{k,q}$ remain identical, and the trust-region subproblem \eqref{eq:PSTrustRegionProb} reduces to an optimization problem within a spherical region. We solve it using the truncated conjugate gradient method with a boundary improvement step, as implemented in \texttt{COBYQA} \cite{COBYQA}. This variant of \texttt{UPOQA} is labeled \texttt{UPOQA (non-struct)}. 

\begin{figure}[ht]
    \centering
    \subfloat[$\varepsilon = 10^{-1}$]{\includegraphics[width = 0.43\textwidth]{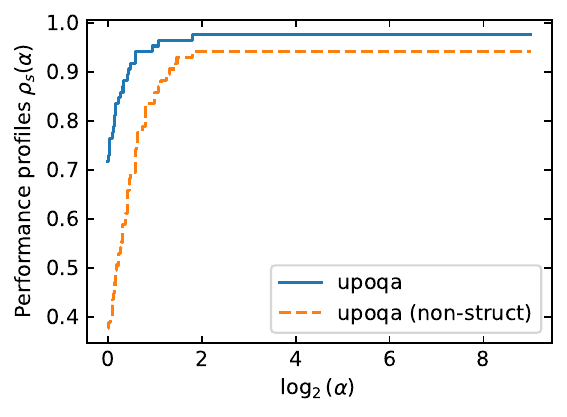}}
    \hfill
    \subfloat[$\varepsilon = 10^{-5}$]{\includegraphics[width = 0.43\textwidth]{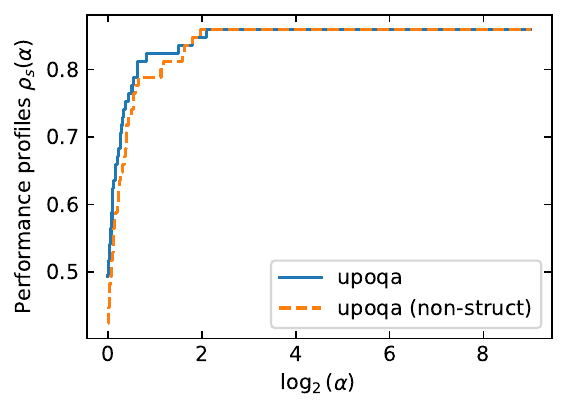}}
    \caption{
        Performance profiles of \texttt{UPOQA} and its unstructured-trust-region variant \texttt{UPOQA (non-struct)} on the test problems from Table \ref{tab:cutest_probs} in Appendix \ref{app:probs}, with convergence tolerances $\varepsilon = 10^{-1}$ and $10^{-5}$.
    }
    \Description{Fully described in the text.}
    \label{fig:non-struct_vs_struct}
\end{figure}

Figure \ref{fig:non-struct_vs_struct} presents performance profiles comparing both algorithms under convergence tolerances of $\varepsilon = 10^{-1}$ and $\varepsilon = 10^{-5}$. For reference, \texttt{UPOQA (single)}, \texttt{L-BFGS-B (ffd)}, and \texttt{NEWUOA} were included in the profile computations to determine the least function values and minimum evaluation numbers required for convergence, though their curves are omitted for clarity.  At $\varepsilon = 10^{-1}$, \texttt{UPOQA} significantly outperforms \texttt{UPOQA (non-struct)}. With structured trust regions, \texttt{UPOQA} achieves the fastest convergence on 71.8\% of the problems, compared to only 36.5\% for \texttt{UPOQA (non-struct)}. Note that the sum of these proportions exceeds 100\% because both solvers require identical evaluation numbers to converge on a small subset of problems. Additionally, even after exhausting the evaluation budget, \texttt{UPOQA (non-struct)} exhibits a 3.5\% lower success rate than \texttt{UPOQA}. At $\varepsilon = 10^{-5}$, the performance gap between the two algorithms narrows considerably, with their profile curves becoming closely matched. Nevertheless, \texttt{UPOQA} maintains a slight speed advantage, converging first on 48.2\% of the problems versus 42.4\% for \texttt{UPOQA (non-struct)}.

These results indicate that \texttt{UPOQA} with structured trust regions maintains a measurable advantage. This benefit is primarily found during the initial optimization phase or when only low-accuracy solutions are required. When high-accuracy solutions are needed, the \texttt{UPOQA (non-struct)} algorithm remains competitive.

\end{document}

%% file: setup.tex
\usepackage[ruled,vlined,linesnumbered]{algorithm2e}
\usepackage{multirow}
\usepackage{fancyvrb}
\usepackage{listings}
\usepackage{array}
\usepackage{makecell}
\usepackage{hyperref}
\usepackage{color}
\usepackage{epstopdf}
\usepackage[caption=false]{subfig}

\newcommand{\llIf}[2]{{\let\par\relax\lIf{#1}{#2}}}
\newcommand{\llElse}[1]{{\let\par\relax\lElse{#1}}}
\newcommand{\from}{\leftarrow}

\newcommand{\true}{\textbf{true}}

\newcommand{\Break}{\textbf{Break}}
\newcommand{\false}{\textbf{false}}
\newcommand{\argmin}{\mathop{\text{argmin}}}

\newcommand{\Real}{{\mathbb{R}}}

